\documentclass{amsart}
\usepackage[a4paper,marginratio=1:1]{geometry}
\usepackage{amssymb,mathrsfs,amsmath,braket}
\usepackage[initials,non-sorted-cites]{amsrefs}
\usepackage[symbol,perpage]{footmisc}
\usepackage{comment}

\numberwithin{equation}{section}
\usepackage{mathtools}\mathtoolsset{showonlyrefs} 

\usepackage{xcolor}

\newtheorem{thm}{Theorem}[section]
\newtheorem{prop}[thm]{Proposition}
\newtheorem{lem}[thm]{Lemma}

\theoremstyle{definition}
\newtheorem{dfn}[thm]{Definition}
\theoremstyle{remark}
\newtheorem{rem}[thm]{Remark}
\newtheorem{ex}[thm]{Example}

\DeclareMathOperator{\Real}{Re} 
\DeclareMathOperator{\Imaginary}{Im} 
\DeclareMathOperator{\Id}{id} 
\DeclareMathOperator{\Ker}{Ker} 
\DeclareMathOperator{\proj}{pr}

\def\bbR{\mathbb{R}} 
\def\bbC{\mathbb{C}} 
\def\bbZ{\mathbb{Z}} 
\def\bbH{\mathbb{H}} 

\newcommand{\abs}[1]{\lvert #1 \rvert} 

\title[Heisenberg Bieberbach manifolds]{Joint spectrum of the sub-Laplacian and the Reeb vector field on three-dimensional Heisenberg Bieberbach manifolds}

\author{Yoshiaki Suzuki}
\address{Department of Mathematics, Graduate School of Science, The University of Osaka, 1-1 Machikaneyama-cho, Toyonaka, Osaka, 560-0043, Japan}
\email{suzuki.yoshiaki.sci@osaka-u.ac.jp, yoshiaki.suzuki.math@gmail.com}

\author{Yuya Takeuchi}
\address{Division of Mathematics, Institute of Pure and Applied Sciences, University of Tsukuba, Tsukuba, Ibaraki 305-8571, Japan}
\email{ytakeuchi@math.tsukuba.ac.jp, yuya.takeuchi.math@gmail.com}



\thanks{Y.S. is partly supported by JSPS KAKENHI Grant Number 24K06738. Y.T. was supported by JSPS KAKENHI Grant Number JP25K17247. 
}

\begin{document}

\begin{abstract}
    A Heisenberg Bieberbach manifold is a compact quotient of the Heisenberg group by a discrete torsion-free subgroup of its pseudo-Hermitian automorphism group.
    In this paper, we study the joint spectrum of the sub-Laplacian and the Reeb vector field on three-dimensional Heisenberg Bieberbach manifolds.
    In particular, we explicitly compute the multiplicities of the joint spectrum using theta functions and the character theory of finite groups.
\end{abstract}

\maketitle

\section{Introduction}

The spectrum of differential operators on manifolds is one of the most fundamental topics in geometric analysis.
In CR and pseudo-Hermitian geometry, one is naturally led to study operators such as the sub-Laplacian $\Delta_{b}$ and the Reeb vector field $T$, which play important roles.
If a pseudo-Hermitian manifold is Sasakian, or equivalently, if $T$ preserves the CR structure, then $\Delta_{b}$ commutes with $T$.
In particular, we can consider the joint spectrum of the self-adjoint operators $\Delta_{b}$ and $i^{- 1} T$.

The Heisenberg group $\bbH$ is the flat model of CR manifolds.
A \emph{Heisenberg Bieberbach manifold} is a compact quotient of $\bbH$ by a discrete torsion-free subgroup $\Gamma$ of the pseudo-Hermitian automorphism group of $\bbH$, which we call a \emph{Heisenberg Bieberbach group}.
Equivalently, a Heisenberg Bieberbach manifold is a closed CR manifold whose universal cover is CR diffeomorphic to $\bbH$.
The standard contact form $\theta_{\bbH}$ on $\bbH$ descends to a flat contact form on $\Gamma \backslash \bbH$.
Note that this contact form is canonical from the viewpoint of the CR Yamabe problem; any contact form on $\Gamma \backslash \bbH$ with constant Tanaka--Webster scalar curvature must be homothetic to this flat contact form~\cite{Jerison-Lee1987Yamabe}*{Theorem 7.1}.
It is therefore natural to study spectral properties on such manifolds.
In this paper, we study the joint spectrum of $\Delta_{b}$ and $i^{- 1} T$ on three-dimensional Heisenberg Bieberbach manifolds.

To this end, we first consider a lattice $\Lambda$ in $\bbH$, which acts on $\bbH$ by left translations; in this case, $\Lambda \backslash \bbH$ is called a \emph{Heisenberg nilmanifold}.
Folland~\cite{folland2004compact} studied the joint spectrum of $\Delta_{b}$ and $i^{- 1} T$ on $\Lambda \backslash \bbH$.
He first reduced the lattice to a normal form by a suitable group automorphism of $\bbH$.
Then the Weil--Brezin transform yields a family of joint eigenfunctions of $\Delta_{b}$ and $i^{- 1} T$.

Let $\Gamma$ be a Heisenberg Bieberbach group.
Then $\Gamma$ contains a lattice $\Lambda$ in $\bbH$ such that $\Lambda \backslash \bbH$ is a finite cover of $\Gamma \backslash \bbH$.
This implies that a joint eigenfunction of $\Delta_{b}$ and $i^{- 1} T$ on $\Gamma \backslash \bbH$ can be considered as a $(\Gamma / \Lambda)$-invariant joint eigenfunction on $\Lambda \backslash \bbH$.
Using this idea, the first author~\cite{Suzuki2024preprint} computed the joint spectrum for some certain Heisenberg Bieberbach manifolds $\Gamma \backslash \bbH$ satisfying $\# (\Gamma / \Lambda) = 2$ or $4$.

In this paper, we consider a general three-dimensional Heisenberg Bieberbach manifold $\Gamma \backslash \bbH$.
Our main result is a classification of $\Gamma$ up to conjugacy in the CR automorphism group of $\bbH$ together with an explicit description of the multiplicities of the joint spectrum of $\Delta_{b}$ and $i^{- 1} T$.
The difficulty in applying Folland's method in the general case lies in the fact that the group automorphism that sends $\Gamma$ to a normal form does not preserve the CR structure on $\bbH$, and hence does not preserve $\Delta_{b}$.
To overcome this difficulty, our approach is based on a classification of lattices in $\bbH$ up to conjugacy in the CR automorphism group of $\bbH$.
Then we construct joint eigenfunctions on a Heisenberg nilmanifold explicitly using theta functions.
The multiplicities of the joint spectrum are computed by applying the character theory to the representation of $\Gamma / \Lambda$ on joint eigenspaces on $\Lambda \backslash \bbH$.
In the course of the proof, we also classify Heisenberg Bieberbach groups up to conjugacy in the CR automorphism group of $\bbH$.

This paper is organized as follows.
In Section~\ref{section:CR manifolds}, we recall basic material on CR manifolds.
In Section~\ref{section:HB manifolds}, we review the definitions and some properties of Heisenberg Bieberbach groups and Heisenberg Bieberbach manifolds.
In Section~\ref{section:joint spectrum on HN}, we classify lattices $\Lambda$ in $\bbH$ up to conjugacy in the CR automorphism group of $\bbH$ and construct joint eigenfunctions on Heisenberg nilmanifolds.
In the remaining sections, we classify $\Gamma$ according to the order of $\Gamma / \Lambda$ and compute the multiplicities of the joint spectrum.

\section{CR manifolds}
\label{section:CR manifolds}

In this section, we recall the definition and fundamental concepts of CR manifolds.
We restrict our attention to the three-dimensional case.
Let $M$ be a three-dimensional manifold. A \emph{CR structure} $T^{1,0} M$ on $M$ is a complex line subbundle of $\bbC TM$ such that
\begin{equation}
    T^{1, 0} M \cap T^{0, 1} M = \{0\},
\end{equation}
where $T^{0, 1} M = \overline{T^{1, 0} M}$.
A CR manifold $(M,T^{1, 0} M)$ is said to be \emph{strictly pseudoconvex} if there exists a nowhere-vanishing real one-form $\theta$ on $M$ such that $\theta$ annihilates $T^{1, 0} M$ and 
the \emph{Levi form} $L_{\theta}$ defined by
\begin{equation}
    L_{\theta}(Z,W) \coloneqq -i\, d\theta (Z,\overline{W}),\qquad Z,W \in T^{1, 0} M
\end{equation}
is a positive definite Hermitian form on $T^{1, 0} M$. 
In this situation, such $\theta$ is called a \emph{contact form}, and the triple $(M, T^{1, 0} M, \theta)$ is called a \emph{pseudo-Hermitian manifold}.
For a contact form $\theta$, the unique vector field $T$ satisfying
\begin{equation}
    \theta(T) = 1, \quad T \lrcorner\, d\theta =0
\end{equation}
is called the \emph{Reeb vector field} associated with $\theta$.

Let $(M, T^{1, 0} M, \theta)$ be a pseudo-Hermitian manifold and set $H M \coloneqq \Ker \theta$.
We define an operator $d_{b} \colon C^{\infty}(M) \longrightarrow \Gamma((H M)^{\ast})$ by 
\begin{equation}
     d_{b} u \coloneqq (d u)|_{H M}
\end{equation}
for $u \in C^{\infty}(M)$.
The Levi form $L_{\theta}$ and the volume form $\theta \wedge d \theta$ induce an $L^{2}$-inner product on $\Gamma((H M)^{\ast})$, which determines the formal adjoint $d_{b}^{\ast}$ of $d_{b}$.
The \emph{sub-Laplacian} $\Delta_{b}$ is defined by
\begin{equation}
    \Delta_{b} u \coloneqq d_{b}^{\ast} d_{b} u
\end{equation}
for $u \in C^{\infty}(M)$.
It is known that the sub-Laplacian is subelliptic;
see \cite{Ponge2008HC}*{Section 3.5.3} for example.
In particular if $M$ is closed, then the spectrum of $\Delta_{b}$ is discrete and each eigenspace is a finite dimensional subspace of $C^{\infty}(M)$.

A pseudo-Hermitian manifold $(M, T^{1, 0} M, \theta)$ is said to be \emph{Sasakian} if the Reeb vector field $T$ preserves the CR structure $T^{1, 0} M$; that is, $[T, Z] \in \Gamma(T^{1, 0} M)$ for all $Z \in \Gamma(T^{1, 0} M)$.
Since $T$ preserves both $T^{1, 0} M$ and $\theta$, the sub-Laplacian $\Delta_{b}$ commutes with $T$.
Assume that $M$ is closed.
Then $i^{- 1} T$ induces a self-adjoint operator on each eigenspace of $\Delta_{b}$.
Therefore, the operators $\Delta_{b}$ and $i^{- 1} T$ admit a joint spectrum.

Let $(M, T^{1, 0} M)$ and $(M^{\prime}, T^{1, 0} M^{\prime})$ be CR manifolds.
A smooth map $F \colon M \longrightarrow M^{\prime}$ is called a \emph{CR map} if $F_{\ast} (T^{1, 0} M) \subset T^{1, 0} M^{\prime}$.
If a CR map $F$ is a diffeomorphism, it is called a \emph{CR diffeomorphism}.
In particular, a \emph{CR automorphism} of $M$ is a CR diffeomorphism from $M$ onto itself.
Suppose that $M$ and $M^{\prime}$ are strictly pseudoconvex, and let $\theta$ (resp.\ $\theta^{\prime}$) be a contact form on $M$ (resp.\ $M^{\prime}$).
A CR diffeomorphism $F \colon M \longrightarrow M^{\prime}$ is called a \emph{pseudo-Hermitian diffeomorphism} if it preserves the contact forms; that is, $F^{\ast} \theta^{\prime} = \theta$.
In particular, a \emph{pseudo-Hermitian automorphism} of $M$ is a pseudo-Hermitian diffeomorphism from $M$ onto itself.

\section{The Heisenberg group and Heisenberg Bieberbach manifolds}
\label{section:HB manifolds}

We review some of the standard facts on the Heisenberg group and Heisenberg Bieberbach manifolds.
The three-dimensional \emph{Heisenberg group} is the Lie group $\bbH = \bbC \times \bbR$ with
the group multiplication given by
\begin{equation}
    (z, t) \cdot (z', t') = (z + z', t + t' + 2 \Imaginary (z \bar{z}'))
\end{equation}
for $(z, t),(z', t') \in \bbH$.
We introduce a left-invariant complex vector field $Z$ by
\begin{equation}
    Z \coloneqq \frac{\partial}{\partial z} + i \bar{z} \frac{\partial}{\partial t}.
\end{equation}
The subbundle of $\bbC T \bbH$ spanned by $Z$ defines the canonical CR structure $T^{1, 0} \bbH$ on $\bbH$.
Define a left-invariant one-form $\theta_{\bbH}$ on $\bbH$ by
\begin{equation}
    \theta_{\bbH} \coloneqq d t + i (z d \bar{z} - \bar{z} d z).
\end{equation}
Then $\theta_{\bbH}$ annihilates $T^{1, 0} \bbH$ and satisfies $- i \, d \theta_{\bbH}(Z, \bar{Z}) = 2 > 0$;
in particular, $(\bbH, T^{1, 0} \bbH)$ is a strictly pseudoconvex CR manifold and $\theta_{\bbH}$ is a contact form on $\bbH$.
The Reeb vector field $T_{\bbH}$ associated with $\theta_{\bbH}$ is given by $T_{\bbH} = \partial / \partial t$.
Note that the pseudo-Hermitian manifold $(\bbH, T^{1, 0} \bbH, \theta_{\bbH})$ is Sasakian since $[T_{\bbH}, Z] = 0$.

The action $\mu$ of $\bbC^{\ast}$ on $\bbH$ is defined by
\begin{equation}
    \mu(\lambda)(z, t) \coloneqq (\lambda z, \abs{\lambda}^{2} t).
\end{equation}
This action preserves the CR structure $T^{1, 0} \bbH$ of $\bbH$ and satisfies $\mu(\lambda)^{\ast} \theta_{\bbH} = \abs{\lambda}^{2} \theta_{\bbH}$.
In particular, $\mu(\lambda)$ is a pseudo-Hermitian automorphism if and only if $\abs{\lambda} = 1$.

We consider the semidirect product $\bbH \rtimes \bbC^{\ast}$; that is, the direct product $\bbH \times \bbC^{\ast}$ with the group multiplication
\begin{equation}
    (g, \lambda) \cdot (h, \nu) \coloneqq (g \cdot \mu(\lambda)(h), \lambda \nu)
\end{equation}
for $(g,\lambda),(h,\nu) \in \bbH \times \bbC^*$.
The group $\bbH \rtimes \bbC^{\ast}$ acts on $\bbH$ from the left by
\begin{equation}
    (g, \lambda) h \coloneqq g \cdot \mu(\lambda)(h),
\end{equation}
which preserves the CR structure $T^{1, 0} \bbH$.
In what follows, we identify $(g, \lambda)$ with the automorphism $g \mu(\lambda)$ on $\bbH$.
The CR automorphism group of $(\bbH, T^{1, 0} \bbH)$ is known to be equal to $\bbH \rtimes \bbC^{\ast}$~\cite{burns-shnider1976spherical}.
Moreover, the pseudo-Hermitian automorphism group of $(\bbH, T^{1, 0} \bbH, \theta_{\bbH})$ coincides with $\bbH \rtimes S^{1}$.

Let $\Gamma$ be a discrete subgroup of $\bbH \rtimes \bbC^{\ast}$.
It is known that $\Gamma$ must be a subgroup of $\bbH \rtimes S^{1}$ if $\Gamma$ acts properly discontinuously on $\bbH$~\cite{burns-shnider1976spherical}*{Proposition 5.6}.
In view of this fact, we consider the following class of subgroups.

\begin{dfn}
A subgroup $\Gamma \subset \bbH \rtimes S^{1}$ is called a \emph{Heisenberg Bieberbach group} if $\Gamma$ is a discrete torsion-free subgroup and the quotient $\Gamma \backslash \bbH$ of $\bbH$ by $\Gamma$ is compact.
In this situation, $\Gamma \backslash \bbH$ is called a \emph{Heisenberg Bieberbach manifold}.
\end{dfn}

Note that a discrete subgroup of $\bbH \rtimes S^{1}$ acts freely on $\bbH$ if and only if it is torsion-free.
A detailed discussion in \cite{charlap1986bieberbach}*{Section 1.1}, while given in the Euclidean case, can also be applied to the Heisenberg setting.

\begin{ex}
A co-compact discrete subgroup $\Lambda \subset \bbH$ is called a \emph{lattice}, which is an example of a Heisenberg Bieberbach group. 
We call $\Lambda \backslash \bbH$ a \emph{Heisenberg nilmanifold}.
Let $L(\Lambda)$ be the image of a lattice $\Lambda$ by the first projection $\bbH \to \bbC$. 
Then $L(\Lambda)$ is a lattice in $\bbC$.
Moreover, the center $Z(\Lambda)$ of $\Lambda$ must be non-trivial and coincides with $\Lambda \cap Z(\bbH)$.
Furthermore, $N \coloneqq \# (Z(\Lambda) / [\Lambda, \Lambda])$ must be finite.
See \cite{folland2004compact}*{Section 2} for details.
\end{ex}

Let $\Gamma$ be a Heisenberg Bieberbach group and set $\Lambda \coloneqq \Gamma \cap \bbH$.
Then $\Lambda$ is a lattice in $\bbH$ and a normal subgroup of $\Gamma$. Moreover, the quotient group $\Gamma / \Lambda$ is finite~\cite{burns-shnider1976spherical}*{Proposition 5.6}.
Note that $N = \# (Z(\Lambda) / [\Lambda, \Lambda])$ is invariant under conjugation in $\bbH \rtimes \bbC^{\ast}$.
Denote by $\proj_{2} \colon \bbH \rtimes \bbC^{\ast} \to \bbC^{\ast}$ the second projection, which is a group homomorphism.

\begin{lem}
\label{lem:rotation}
    The subgroup $\proj_{2}(\Gamma)$ in $S^{1}$ is one of $\{1\}$, $\langle - 1 \rangle$, $\langle i \rangle$, $\langle e^{2 \pi i / 3} \rangle$ or $\langle e^{\pi i / 3} \rangle$.
\end{lem}

\begin{proof}
    Let $e^{i \theta} \in \proj_{2}(\Gamma)$ and take $g \in \bbH$ such that $\gamma \coloneqq (g, e^{i \theta}) \in \Gamma$.
    Since $\Lambda$ is invariant under conjugation by $\gamma$, the multiplication by $e^{i \theta}$ induces an automorphism of the lattice $L(\Lambda)$ in $\bbC$.
    In particular, its trace as an $\bbR$-linear map must be an integer.
    Hence we have $2 \cos \theta \in \bbZ$.
    This implies $\theta \equiv 0, \pi / 3, \pi / 2, 2 \pi / 3$ modulo $\pi$, which gives the classification of $\proj_{2}(\Gamma)$.
\end{proof}

Let $\Gamma$ be a Heisenberg Bieberbach group.
Since the action of $\Gamma$ on $\bbH$ preserves both the CR structure $T^{1, 0} \bbH$ and the contact form $\theta_{\bbH}$, the Heisenberg Bieberbach manifold $\Gamma \backslash \bbH$ has a strictly pseudoconvex CR structure $T^{1, 0} (\Gamma \backslash \bbH)$ and a contact form on $\Gamma \backslash \bbH$; we use the same symbol $\theta_{\bbH}$ by abuse of notation.
Let $\Gamma'$ be another Heisenberg Bieberbach group.
If there exists $\gamma \in \bbH \rtimes \bbC^{\ast}$ such that $\gamma \Gamma \gamma^{- 1} = \Gamma'$, then $\gamma$ induces a CR diffeomorphism from $\Gamma \backslash \bbH$ to $\Gamma' \backslash \bbH$.

\section{Joint spectrum on Heisenberg nilmanifolds}
\label{section:joint spectrum on HN}

For $\tau \in \bbC$ and $c \in \bbR_{> 0}$ satisfying $4 \Imaginary \tau / c \in \bbZ_{> 0}$, we set
\begin{equation}
    \Lambda_{\tau, c} \coloneqq \langle  (1, 0), (\tau, 0), (0, c) \rangle = \Set{(m + n \tau, 2 m n \Imaginary \tau + p c) | m, n, p \in \bbZ} \subset \bbH,
\end{equation}
which is a lattice in $\bbH$.
The center $Z(\Lambda_{\tau, c})$ of $\Lambda_{\tau, c}$ is given by $\Set{(0, n c) | n \in \bbZ}$.
Note that $N = 4 \Imaginary \tau / c$ since
\begin{equation}
    [\Lambda_{\tau, c}, \Lambda_{\tau, c}] = \Set{(0, 4 n \Imaginary \tau ) | n \in \bbZ}.
\end{equation}
We also note that $L(\Lambda_{\tau, c}) = \langle 1, \tau \rangle \eqqcolon L_{\tau}$.
The following proposition shows that any lattice in $\bbH$ is conjugate to some $\Lambda_{\tau, c}$ in $\bbH \rtimes \bbC^{\ast}$.

\begin{prop}
\label{prop:lattice}
    Let $\Lambda$ be a lattice in $\bbH$. Then there exist $\gamma \in \bbH \rtimes \bbC^{\ast}$ and $\tau \in \bbC$ such that $\Imaginary \tau > 0$ and $\gamma \Lambda \gamma^{- 1} = \Lambda_{\tau, c}$,
    where $c = 4 \Imaginary \tau / N$.
\end{prop}

\begin{proof}
    Since $L(\Lambda)$ is a lattice in $\bbC$,
    we can take a basis $\omega_{1}, \omega_{2}$ of $L(\Lambda)$ satisfying $\Imaginary (\overline{\omega}_{1} \omega_{2}) > 0$.
    Set $\gamma_{1} \coloneqq (0, \omega_{1}^{-1}) \in \bbH \rtimes \bbC^{\ast}$.
    Then $\Lambda_{1} \coloneqq \gamma_{1} \Lambda \gamma_{1}^{-1}$ is a lattice in $\bbH$ such that $L(\Lambda_{1}) = L_{\tau}$ for some $\tau \in \bbC$ with $\Imaginary \tau > 0$.
    The center of $\Lambda_{1}$ is generated by $(0, c)$ for some $c > 0$.
    Moreover, $\Lambda_{1}$ is generated by three elements $(1, s_{1}), (\tau, s_{2}), (0, c)$ for some $s_{1}, s_{2} \in \bbR$.
    Set
    \begin{equation}
        a \coloneqq \frac{s_{2} - s_{1} \Real \tau}{4 \Imaginary \tau}, \qquad b \coloneqq - \frac{s_{1}}{4}.
    \end{equation}
    Then $g_{2} \coloneqq (a + b i, 0) \in \bbH$ and one can see that
    \begin{equation}
        \Lambda_{2} \coloneqq g_{2} \Lambda_{1} g_{2}^{-1} = \langle  (1, 0), (\tau, 0), (0, c) \rangle.
    \end{equation}
    Moreover, we have
    \begin{equation}
        (\tau, 0) (1, 0) (\tau, 0)^{-1} (1, 0)^{-1} = (0, 4 \Imaginary \tau) \in \Lambda_2 \cap Z(\bbH),
    \end{equation}
    which implies $4 \Imaginary \tau / c \in \bbZ_{> 0}$.
    Therefore, $\gamma \coloneqq g_{2} \gamma_{1} \in \bbH \rtimes \bbC^{\ast}$ satisfies $\gamma \Lambda \gamma^{-1} = \Lambda_{\tau, c}$.
\end{proof}

We also give a classification of elements in $\bbH$ that normalizes $\Lambda_{\tau, c}$ for later use,

\begin{prop}
\label{prop:conjugate}
    An element $g = (\zeta, \sigma) \in \bbH$ satisfies $g \Lambda_{\tau, c} g^{-1} = \Lambda_{\tau, c}$ if and only if $\zeta = (k + l \tau) / N$ for some $k, l \in \bbZ$.
\end{prop}

\begin{proof}
    A computation yields
    \begin{equation}
        g (1, 0) g^{-1} = (1, 0) (0, 4 \Imaginary \zeta), \quad g (\tau, 0) g^{-1}  = (\tau, 0) (0, 4 \Imaginary (\zeta \bar{\tau})), \quad g (0, c) g^{-1} = (0, c).
    \end{equation}
    Hence $g \Lambda_{\tau, c} g^{-1} \subset \Lambda_{\tau, c}$ if and only if $4 \Imaginary \zeta, 4 \Imaginary (\zeta \bar{\tau}) \in c \bbZ$,
    and so there exist $k, l \in \bbZ$ such that
    \begin{equation}
        \zeta = \frac{c}{4 \Imaginary \tau} (k + l \tau) = \frac{k + l \tau}{N}.
    \end{equation}
    Assume that $\zeta$ is of this form.
    Then a similar argument shows that $g^{-1} \Lambda_{\tau, c} g \subset \Lambda_{\tau, c}$, which means $\Lambda_{\tau, c} \subset g \Lambda_{\tau, c} g^{-1}$.
\end{proof}

Let $c, r \in \bbR_{> 0}$ and $\tau \in \bbC$ be such that $4 \Imaginary \tau / c \in \bbZ_{> 0}$. 
We consider the pseudo-Hermitian manifold $(M_{\tau, c} \coloneqq \Lambda_{\tau, c} \backslash \bbH, T^{1, 0} M_{\tau, c}, r \theta_{\bbH})$.
Recall that $N = 4 \Imaginary \tau / c$.
To simplify notation, we omit the subscript ${}_{\tau, c}$ in this section.
The sub-Laplacian $\Delta_{b}$ and the Reeb vector field $T$ is given by
\begin{equation}
    \Delta_{b} = - \frac{1}{2 r} (Z \bar{Z} + \bar{Z} Z), \qquad T = \frac{1}{r} T_{\bbH}.
\end{equation}
The Heisenberg group $\bbH$ has the standard volume form $\theta_{\bbH} \wedge d \theta_{\bbH}$, which is bi-invariant and thus induces a measure on $M$.
We can identify the space $L^{2}(M)$ of square integrable functions on $M$ with the space of $\Lambda$-invariant functions $f$ on $\bbH$ with $\int_{Q} \abs{f}^{2} \, \theta_{\bbH} \wedge d \theta_{\bbH} < \infty$, where $Q$ is a fundamental domain for $\Lambda$.
The right action of $\bbH$ on $M$ induces the strongly continuous unitary action $R$ of $\bbH$ on $L^{2}(M)$ defined by
\begin{equation}
    (R(h) f)([g]) \coloneqq f([g h]), \qquad g, h \in \bbH.
\end{equation}
Note that
\begin{equation}
    d R(Z) = Z, \qquad d R(\bar{Z}) = \bar{Z}
\end{equation}
by the definition of $R$.
We would like to decompose $R$ into irreducible unitary representations of $\bbH$.

To this end,
we first expand $f(z, t) \in L^{2}(M)$ into the Fourier series in $t$.
It follows from the $\Lambda$-invariance of $f$ that $f(z, t + c) = f(z, t)$, and hence $f$ can be expanded as follows:
\begin{equation}
    f(z, t) = \sum_{n \in \bbZ} e^{2 \pi i n t / c} g_{n}(z).
\end{equation}
This implies that
\begin{equation}
    L^{2}(M) = \bigoplus_{n \in \bbZ} \mathscr{H}_{n},
\end{equation}
where $\mathscr{H}_{n}$ is the subspace of $L^{2}(M)$ consisting of functions of the form $f(z, t) = e^{2 \pi i n t / c} g(z)$.
Note that $i^{- 1} T f = (2 \pi n / c r) f$ for any $f \in \mathscr{H}_{n}$.

A function in $\mathscr{H}_{0}$ can be regarded as a function on the torus $\bbC / L(\Lambda)$.
Denote by $L(\Lambda)^{\prime}$ the dual lattice of $L(\Lambda)$; that is,
\begin{equation}
    L(\Lambda)^{\prime} = \Set{\zeta \in \bbC | \Real(\zeta \bar{z}) \in \bbZ \text{ for all } z \in L(\Lambda)}.
\end{equation}
For each $\zeta \in L(\Lambda)^{\prime}$,
the function $\chi_{\zeta}(z) \coloneqq e^{2 \pi i \Real(\zeta \bar{z})}$ is an eigenfunction of $\Delta_{b}$ with eigenvalue $\pi^{2} \abs{\zeta}^{2} / r$,
and $\mathscr{H}_{0}$ has the orthogonal decomposition
\begin{equation}
    \mathscr{H}_{0} = \bigoplus_{\zeta \in L(\Lambda)^{\prime}} \bbC \chi_{\zeta}.
\end{equation}
Set $E_{\kappa} \coloneqq \Set{\zeta \in L(\Lambda)^{\prime} | \pi^{2} \abs{\zeta}^{2} / r = \kappa }$ for each $\kappa \geq 0$.
The multiplicity of the eigenvalue $(\kappa, 0)$ in the joint spectrum of $\Delta_{b}$ and $i^{- 1} T$ is given by $\# E_{\kappa}$.

In the case $n \neq 0$, the Stone--von Neumann theorem yields that $R|_{\mathscr{H}_{n}}$ is unitary equivalent to a sum of copies of the Schr\"{o}dinger representation $\pi_{n / c}$ of $\bbH$ in $L^{2}(\bbR)$~\cite{Folland1989Harmonic}*{Theorem 1.50}.
Here, the \emph{Schr\"{o}dinger representation} $\pi_{\beta}$ for $\beta \in \bbR \setminus \{0\}$ is defined by
\begin{equation}
    (\pi_{\beta}(x + i y, t) u)(\xi) = e^{2 \pi i \beta (t - 2 x y - 2 y \xi)} u(\xi + 2 x).
\end{equation}
One can see that
\begin{equation}
    d \pi_{\beta}(Z) = \frac{d}{d \xi} - 2 \pi \beta \xi, \quad d \pi_{\beta}(\bar{Z}) = \frac{d}{d \xi} + 2 \pi \beta \xi, \quad d \pi_{\beta}(\Delta_{b}) = - \frac{1}{r} \frac{d^{2}}{d \xi^{2}} + \frac{4 \pi^{2} \beta^{2}}{r} \xi^{2}. 
\end{equation}
In particular, $d \pi_{\beta}(\Delta_{b})$ is the Hamiltonian of a quantum harmonic oscillator.
Assume that $\beta > 0$.
Then $d \pi_{\beta}(Z)$ (resp.\ $d \pi_{\beta}(\bar{Z})$) corresponds to the creation (resp.\ annihilation) operator.
The kernel of $d \pi_{\beta}(\bar{Z})$ is given by
\begin{equation}
    \mathcal{H}_{0}^{\beta} \coloneqq \Ker d \pi_{\beta}(\bar{Z}) = \bbC e^{- \pi \beta \xi^{2}};
\end{equation}
in particular, this space is one-dimensional.
Moreover, $\mathcal{H}_{k}^{\beta} \coloneqq d \pi_{\beta}(Z^{k}) (\mathcal{H}_{0}^{\beta}) \subset L^{2}(\bbR)$ is also one-dimensional for any $k \in \bbZ_{\geq 0}$, and
\begin{equation}
    L^{2}(\bbR) = \bigoplus_{k \in \bbZ_{\geq 0}} \mathcal{H}_{k}^{\beta}.
\end{equation}
Furthermore, $d \pi_{\beta}(\Delta_{b})$ on $\mathcal{H}_{k}^{\beta}$ is equal to the multiplication by $(4 k + 2) \pi \beta / r$.

Assume that $n > 0$.
Since $R|_{\mathscr{H}_{n}}$ is unitary equivalent to a sum of copies of the Schr\"{o}dinger representation $\pi_{n / c}$, the map $Z^{k} \colon \Ker \bar{Z} \cap \mathscr{H}_{n} \to \mathscr{H}_{n}$ is injective, and
\begin{equation}
     \mathscr{H}_{n} = \bigoplus_{k \in \bbZ_{\geq 0}} \mathscr{H}_{n, k},
\end{equation}
where $\mathscr{H}_{n, k} \coloneqq Z^{k} (\Ker \bar{Z} \cap \mathscr{H}_{n})$.
Moreover, we have
\begin{equation}
    \Delta_{b}|_{\mathscr{H}_{n, k}} = \frac{(4 k + 2) \pi n}{c r} \Id_{\mathscr{H}_{n, k}}.
\end{equation}
We now construct a basis of $\mathscr{H}_{n, k}$ for use in later sections.
To this end, it suffices to construct a basis of $\mathscr{H}_{n, 0}$.
Any $f \in \mathscr{H}_{n, 0}$ is of the form
\begin{equation}
    f(z, t) = \exp \left( \frac{2 \pi n}{c} (i t - \abs{z}^{2} + z^{2}) \right) h(z).
\end{equation}
It follows from the equation $\bar{Z} f = 0$ that
\begin{equation}
    0 = \bar{Z} f (z, t) = \exp \left( \frac{2 \pi n}{c} (i t - \abs{z}^{2} + z^{2}) \right) \frac{\partial h}{\partial \bar{z}}(z),
\end{equation}
which implies that $h$ is a holomorphic function on $\bbC$.
Moreover, since $f$ is $\Lambda$-invariant, we have
\begin{align}
    f(z, t) = f((1, 0) \cdot (z, t)) &= f(z + 1, t - 2 \Imaginary z) \\
    &= \exp \left( \frac{2 \pi n}{c} (i t - \abs{z}^{2} + z^{2}) \right) h(z + 1)
\end{align}
and
\begin{align}
    f(z, t) = f((\tau, 0) \cdot (z, t)) &= f(z + \tau, t - 2 \Imaginary (\bar{\tau} z)) \\
    &= \exp \left( \frac{2 \pi n}{c} (i t - \abs{z}^{2} + z^{2}) \right) \exp \left( 2 \pi i n N z + \pi i n N \tau \right) h(z + \tau);
\end{align}
here we use the fact that $N = 4 \Imaginary \tau / c$.
Hence $h$ satisfies
\begin{equation}
    h(z + 1) = h(z), \qquad h(z + \tau) = \exp (- 2 \pi i n N z - \pi i n N \tau) h(z);
\end{equation}
in particular, $h$ is a $\bbZ$-periodic holomorphic function on $\bbC$.
Thus we have the Fourier expansion
\begin{equation}
    h(z) = \sum_{j = 0}^{n N - 1} \sum_{m \in \bbZ} a_{j, m} \exp \left( 2 \pi i n N \left( m + \frac{j}{n N} \right) z \right).
\end{equation}
Using this expression, we can see that
\begin{equation}
    h(z + \tau) = \sum_{j = 0}^{n N - 1} \sum_{m \in \bbZ} a_{j, m} \exp \left( 2 \pi i n N \left( m + \frac{j}{n N} \right) \tau \right) \exp \left( 2 \pi i n N \left( m + \frac{j}{n N} \right) z \right)
\end{equation}
and
\begin{align}
    \exp (- 2 \pi i n N z - \pi i n N \tau) h(z) &= \sum_{j = 0}^{n N - 1} \sum_{m \in \bbZ} a_{j, m} \exp (- \pi i n N \tau) \exp \left( 2 \pi i n N \left( m - 1 + \frac{j}{n N} \right) z \right) \\
    &= \sum_{j = 0}^{n N - 1} \sum_{m \in \bbZ} a_{j, m + 1} \exp (- \pi i n N \tau) \exp \left( 2 \pi i n N \left( m + \frac{j}{n N} \right) z \right).
\end{align}
Hence $a_{j, m}$ satisfies the recurrence formula
\begin{equation}
    a_{j, m + 1} = \exp \left( 2 \pi i n N \tau \left( m + \frac{1}{2} + \frac{j}{n N} \right) \right) a_{j, m}.
\end{equation}
Solving this recurrence relation, we obtain the following formula:
\begin{equation}
    a_{j, m} = C_{j} \exp \left( \pi i n N \tau \left( m  + \frac{j}{n N} \right)^{2} \right)
\end{equation}
for some $C_{j} \in \bbC$.
This means that
\begin{equation}
    h(z) = \sum_{j = 0}^{n N - 1} C_{j} \sum_{m \in \bbZ} \exp \left( \pi i n N \tau \left( m + \frac{j}{n N} \right)^{2} + 2 \pi i n N \left( m + \frac{j}{n N} \right) z \right).
\end{equation}
For each $0 \leq j \leq n N - 1$, set
\begin{equation}
    h_{j}(z) \coloneqq \sum_{m \in \bbZ} \exp \left( \pi i n N \tau \left( m + \frac{j}{n N} \right)^{2} + 2 \pi i n N \left( m + \frac{j}{n N} \right) z \right)
\end{equation}
and
\begin{equation}
    f_{j}(z, t) \coloneqq \exp \left( \frac{2 \pi n}{c} (i t - \abs{z}^{2} + z^{2}) \right) h_{j}(z) \in \mathscr{H}_{n, 0}.
\end{equation}
Then $(Z^{k} f_{0}, \dots , Z^{k} f_{n N - 1})$ is a basis of $\mathscr{H}_{n, k}$ for any $k \in \bbZ_{\geq 0}$.

In the case $n < 0$,
we have the orthogonal decomposition
\begin{equation}
    \mathscr{H}_{n} = \overline{\mathscr{H}_{- n}} = \bigoplus_{k \in \bbZ_{\geq 0}} \overline{\mathscr{H}_{- n, k}}
\end{equation}
and
\begin{equation}
    \Delta_{b}|_{\overline{\mathscr{H}_{- n, k}}} = \frac{(4 k + 2) \pi \abs{n}}{c r} \Id_{\overline{\mathscr{H}_{- n, k}}}.
\end{equation}
Moreover, $(\bar{Z}^{k} \bar{f}_{0}, \dots , \bar{Z}^{k} \bar{f}_{\abs{n} N - 1})$ is a basis of $\overline{\mathscr{H}_{- n, k}}$ for any $k \in \bbZ_{\geq 0}$.

In summary, we obtain an alternative proof of \cite{folland2004compact}*{Theorem 3.2} that avoids the use of the Weil--Brezin transform.

\begin{thm}
    The joint spectrum of $\Delta_{b}$ and $i^{- 1} T$ on $M$ is
    \begin{equation}
        \Set{\left( \frac{(4 k + 2) \pi \abs{n}}{cr}, \frac{2 \pi n}{c r} \right)|n \in \bbZ \setminus \{0\}, k \in \bbZ_{\geq 0}} \cup \Set{\left( \frac{\pi^{2} \abs{\zeta}^{2}}{r}, 0 \right) | \zeta \in L_{\tau}^{\prime}}.
    \end{equation}
    The multiplicity of $((4 k + 2) \pi \abs{n} / cr, 2 \pi n / cr)$ in the joint spectrum is given by $\abs{n} N$.
    The multiplicity of $(\kappa, 0)$ in the joint spectrum is equal to $\# E_{\kappa}$ for each $\kappa \geq 0$.
\end{thm}

Let $\Gamma$ be a Heisenberg Bieberbach group.
By conjugating in $\bbH \rtimes \bbC^{\ast}$, we may assume that $\Gamma \cap \bbH = \Lambda$.
Then the finite group $\Gamma / \Lambda$ acts on $M$ by pseudo-Hermitian automorphisms, and its quotient is isomorphic to the Heisenberg Bieberbach manifold $\Gamma \backslash \bbH$ as a pseudo-Hermitian manifold.
In particular, the joint spectrum of $\Delta_{b}$ and $i^{- 1} T$ on $\Gamma \backslash \bbH$ is contained in
\begin{equation}
    \Set{\left( \frac{(4 k + 2) \pi \abs{n}}{cr}, \frac{2 \pi n}{c r} \right) | n \in \bbZ \setminus \{0\}, k \in \bbZ_{\geq 0}} \cup \Set{\left( \frac{\pi^{2} \abs{\zeta}^{2}}{r}, 0 \right) | \zeta \in L_{\tau}^{\prime}}.
\end{equation}
Note that the joint eigenspace corresponding to the eigenvalue $(0, 0)$ is the space of constant functions, and so its multiplicity must be $1$.
To simplify notation, we set
\begin{equation}
    \sigma_{n, k} \coloneqq \left( \frac{(4 k + 2) \pi \abs{n}}{cr}, \frac{2 \pi n}{c r} \right)
\end{equation}
for $n \in \bbZ \setminus \{0\}$ and $k \in \bbZ_{\geq 0}$.

\section{Joint spectrum on Heisenberg Bieberbach manifolds $\Gamma \backslash \bbH$ with $\proj_{2}(\Gamma) = \langle - 1 \rangle$}
\label{section:joint spectrum on HB of order 2}

Let $\Gamma$ be a Heisenberg Bieberbach group with $\proj_{2}(\Gamma) = \langle - 1 \rangle$ and consider the Heisenberg Bieberbach manifold $\Gamma \backslash \bbH$.
Recall that $N = \# (Z(\Lambda) / [\Lambda, \Lambda])$, where $\Lambda = \Gamma \cap \bbH$.
We first show that $\Gamma$ is conjugate to the standard one in $\bbH \rtimes \bbC^{\ast}$.

\begin{prop}
\label{prop:HB group of order 2}
    Let $\Gamma$ be a Heisenberg Bieberbach group satisfying $\proj_{2}(\Gamma) = \langle - 1 \rangle$.
    Then $N$ must be even and there exists $\tau \in \bbC$ such that $\Imaginary \tau > 0$ and $\Gamma$ is conjugate in $\bbH \rtimes \bbC^{\ast}$ to
    \begin{equation}
        \Gamma_{\tau, c} \coloneqq \langle \Lambda_{\tau, c}, \varphi \coloneqq (0, c / 2) \mu(- 1) \rangle,
    \end{equation}
    where $c = 4 \Imaginary \tau / N$.
\end{prop}

\begin{proof}
    Without loss of generality, we may assume that $\Gamma \cap \bbH = \Lambda_{\tau, c}$.
    Take $g = (\zeta, \sigma) \in \bbH$ such that $\gamma \coloneqq g \mu(- 1) \in \Gamma$.
    Note that $\Gamma = \langle \Lambda_{\tau, c}, \gamma \rangle$.
    We can see that $2 \sigma \in c \bbZ$ since $\gamma^{2} = (0, 2 \sigma) \in \Gamma \cap \bbH = \Lambda_{\tau, c}$.
    Moreover, it follows from $\gamma \Lambda_{\tau, c} \gamma^{-1} = \Lambda_{\tau, c}$ that
    \begin{align}
        \gamma (1, 0) \gamma^{- 1} &= (- 1, 4 \Imaginary \bar{\zeta})= (- 1, 0) (0, 4 \Imaginary \bar{\zeta}) \in \Lambda_{\tau, c}, \\
        \gamma (\tau, 0) \gamma^{- 1} &= (- \tau, 4 \Imaginary (\tau \bar{\zeta})) = (- \tau, 0) (0, 4 \Imaginary (\tau \bar{\zeta})) \in \Lambda_{\tau, c},
    \end{align}
    which implies that $4 \Imaginary \bar{\zeta}, 4 \Imaginary (\tau \bar{\zeta}) \in c \bbZ$.
    Hence there exist $k, l \in \bbZ$ such that
    \begin{equation}
        \zeta = \frac{c}{4 \Imaginary \tau} (k + l \tau) = \frac{k + l \tau}{N}.
    \end{equation}
    It remains to determine the condition under which $\Gamma$ is torsion-free.
    Every element of $\Lambda_{\tau, c} \setminus \{(0, 0)\}$ has infinite order.
    Any element $\gamma' \in \Gamma \cap \proj_{2}^{-1}(- 1)$ is of the form
    \begin{equation}
        \gamma' = (m + n \tau, 2 m n \Imaginary \tau + p c) \gamma = (\zeta + m + n \tau, \sigma + m n N c/ 2 + (k n - l m) c / 2 + p c) \mu(- 1)
    \end{equation}
    for some $m, n, p \in \bbZ$.
    This  $\gamma'$ is not torsion if and only if
    \begin{equation}
        (\gamma')^{2} = (0, 2 \sigma + m n N c + (k n - l m) c + 2 p c) \neq (0, 0).
    \end{equation}
    Since $p \in \bbZ$ is arbitrary,
    $\Gamma$ is torsion-free if and only if
    \begin{equation}
        \frac{2 \sigma}{c} + m n N + k n - l m
    \end{equation}
    is an odd integer for any $m, n \in \bbZ$.
    In particular, $2 \sigma / c$ must be an odd integer.
    By choosing $m$ and $n$ appropriately, if at least one of $k$ and $l$ is odd, then $2 \sigma / c + m n N + k n - l m$ is even, which is a contradiction.
    Hence both $k$ and $l$ are even.
    A similar argument shows that $N$ must also be even.
    
    Set
    \begin{equation}
        \zeta' \coloneqq - \frac{\zeta}{2} = - \frac{k + l \tau}{2 N}.
    \end{equation}
    It follows from Proposition~\ref{prop:conjugate} that $h \coloneqq (\zeta', 0) \in \bbH$ satisfies $h \Lambda_{\tau, c} h^{-1} = \Lambda_{\tau, c}$.
    Moreover, we have
    \begin{equation}
        h \gamma h^{-1} = (\zeta', 0) (\zeta, \sigma) (\zeta', 0) \mu(- 1) = (0, \sigma) \mu(- 1),
    \end{equation}
    and hence
        \begin{equation}
            h\Gamma h^{-1} = \langle \Lambda_{\tau, c}, (0, \sigma) \mu(- 1) \rangle.
        \end{equation}
    Since $2 \sigma / c$ is odd, $\sigma$ is written as $c(2 q + 1)/2$ for some $q \in \bbZ$, and so we have
        \begin{equation}
            (0,qc)\varphi = (0,c(2 q + 1)/2)\mu(-1) =(0, \sigma)\mu(-1).
        \end{equation}
    Therefore, we get
         \begin{equation}
             \langle \Lambda_{\tau, c}, (0, \sigma) \mu(- 1) \rangle
        = \Gamma_{\tau, c}.
         \end{equation}
    Summarizing the above, $N$ must be even and $\Gamma$ is conjugate in $\bbH \rtimes \bbC^{\ast}$ to $\Gamma_{\tau, c}$.
\end{proof}

This proposition implies that $\Gamma \backslash \bbH$ is isomorphic to $(M \coloneqq \Gamma_{\tau, c} \backslash \bbH, T^{1, 0} M, r \theta_{\bbH})$ for some $r > 0$ as a pseudo-Hermitian manifold.
Hence it suffices to study the joint spectrum of $\Delta_{b}$ and $i^{- 1} T$ on this pseudo-Hermitian manifold.
This was already proved by the first author~\cite{Suzuki2024preprint}*{Section 4}; however we give another proof here using our method.

\begin{thm}
\label{thm:joint spectrum of order 2}
    The multiplicity of $\sigma_{n, k}$ in the joint spectrum of $\Delta_{b}$ and $i^{- 1} T$ on $M$ is given by $\abs{n} N / 2 + (- 1)^{n + k}$.
    The multiplicity of $(\kappa, 0)$ in the joint spectrum is equal to $\# E_{\kappa} / 2$ for any $\kappa > 0$.
\end{thm}

\begin{proof}
    Any function on $M$ can be identified with a $\varphi$-invariant function on $M_{\tau, c}$.
    Denote by $\mathscr{H}_{n}^{\varphi}$ the space of $\varphi$-invariant functions in $\mathscr{H}_{n}$.
    Then $L^{2}(M) = \bigoplus_{n \in\bbZ} \mathscr{H}_{n}^{\varphi}$.
    
    In the case $n = 0$, one has $\mathscr{H}_{0} = \bigoplus_{\zeta \in L_{\tau}^{\prime}} \bbC \chi_{\zeta}$
    and $\Delta_{b} \chi_{\zeta} = (\pi^{2} \abs{\zeta}^{2} / r) \chi_{\zeta}$.
    Since $\varphi^{\ast} \chi_{\zeta} = \chi_{- \zeta}$,
    the multiplicity of the eigenvalue $(\kappa, 0)$ in the joint spectrum of $\Delta_{b}$ and $i^{- 1} T$ on $M$ is equal to $\# E_{\kappa} / 2$ for any $\kappa > 0$.

    Assume that $n > 0$.
    Set $\mathscr{H}_{n, k}^{\varphi} = \mathscr{H}_{n, k} \cap \mathscr{H}_{n}^{\varphi}$, which is the space of $\varphi$-invariant functions in $\mathscr{H}_{n, k}$.
    Note that $\mathscr{H}_{n}^{\varphi} = \bigoplus_{k \in \bbZ_{\geq 0}} \mathscr{H}_{n, k}^{\varphi}$ and
    $\mathscr{H}_{n, k}^{\varphi}$ is the joint eigenspace corresponding to the eigenvalue $\sigma_{n, k}$ of $\Delta_{b}$ and $i^{- 1} T$ on $M$.
    To compute the dimension of $\mathscr{H}_{n, k}^{\varphi}$,
    we consider the representation $\rho$ of $\Gamma_{\tau, c} / \Lambda_{\tau, c} \cong \langle \varphi \rangle \cong \bbZ / 2 \bbZ$ on $\mathscr{H}_{n, k}$ defined by the pullback.
    The Schur orthogonality relation implies that
    \begin{equation}
        \dim \mathscr{H}_{n, k}^{\varphi} = \frac{1}{2} \sum_{l = 0}^{1} \chi_{\rho}(\varphi^{l}) = \frac{1}{2} (n N + \chi_{\rho}(\varphi)),
    \end{equation}
    where $\chi_{\rho}$ is the character of $\rho$.
    Hence it suffices to compute the value $\chi_{\rho}(\varphi)$.
    For each $0 \leq j \leq n N - 1$,
    \begin{align}
        (\varphi^{\ast} f_{j})(z, t)
        &= f_{j}(- z, t + c / 2) \\
        &= \exp \left( \frac{2 \pi n}{c} \left(i t + \frac{i c}{2} - \abs{- z}^{2} + (- z)^{2} \right) \right) h_{j}(- z) \\
        &= (- 1)^{n} \exp \left( \frac{2 \pi n}{c} (i t - \abs{z}^{2} + z^{2}) \right) h_{j}(- z).
    \end{align}
    It follows from the definition of $h_{j}$ that
    \begin{align}
        h_{j}(- z)
        &= \sum_{m \in \bbZ} \exp \left( \pi i n N \tau \left( m + \frac{j}{n N} \right)^{2} + 2 \pi i n N \left( m + \frac{j}{n N} \right) (- z) \right) \\
        &= \sum_{m \in \bbZ} \exp \left( \pi i n N \tau \left( - m - 1 + \frac{n N - j}{n N} \right)^{2} + 2 \pi i n N \left( - m - 1 + \frac{n N - j}{n N} \right) z \right) \\
        &= \sum_{m' \in \bbZ} \exp \left( \pi i n N \tau \left( m' + \frac{n N - j}{n N} \right)^{2} + 2 \pi i n N \left( m' + \frac{n N - j}{n N} \right) z \right) \\
        &= h_{n N - j}(z),
    \end{align}
    where $m' = -m -1$. Here, we adopt the convention that $h_{n N} = h_{0}$.
    Moreover, one can see that
    \begin{equation}
        (\varphi^{\ast} (Z f))(z, t) = (Z f)(- z, t + c / 2) = - (Z (\varphi^{\ast} f))(z, t)
    \end{equation}
    for any $f \in C^{\infty}(\bbH)$.
    Thus we have
    \begin{equation}
        \varphi^{\ast} (Z^{k} f_{j}) = (- 1)^{k} Z^{k} (\varphi^{\ast} f_j) = (- 1)^{n + k} Z^{k} f_{n N - j}
    \end{equation}
    with the convention that $f_{n N} = f_{0}$.
    Since $n N$ is even, $\chi_{\rho}(\varphi) = 2 (- 1)^{n + k}$.
    Therefore we obtain
    \begin{equation}
        \dim \mathscr{H}_{n, k}^{\varphi} = \frac{1}{2} (n N + 2 (- 1)^{n + k}) = \frac{n N}{2} + (- 1)^{n + k}.
    \end{equation}

    In the case $n < 0$, we have the orthogonal decomposition $\mathscr{H}_{n}^{\varphi} = \bigoplus_{k \geq 0} \overline{\mathscr{H}_{- n, k}^{\varphi}}$ and $\overline{\mathscr{H}_{- n, k}^{\varphi}}$ is the joint eigenspace corresponding to the eigenvalue $\sigma_{n, k}$ of $\Delta_{b}$ and $i^{- 1} T$ on $M$.
    In particular,
    \begin{equation}
        \dim \overline{\mathscr{H}_{- n, k}^{\varphi}} = \dim \mathscr{H}_{- n, k}^{\varphi} = \frac{\abs{n} N}{2} + (- 1)^{n + k},
    \end{equation}
    which completes the proof.
\end{proof}

\section{Joint spectrum on Heisenberg Bieberbach manifolds $\Gamma \backslash \bbH$ with $\proj_{2}(\Gamma) = \langle i \rangle$}

In this section, $\equiv$ is understood modulo $4$.
Let $\Gamma$ be a Heisenberg Bieberbach group with $\proj_{2}(\Gamma) = \langle i \rangle$, and consider the Heisenberg Bieberbach manifold $\Gamma \backslash \bbH$.
Recall that $N = \# (Z(\Lambda) / [\Lambda, \Lambda])$, where $\Lambda = \Gamma \cap \bbH$.
We first give a classification of such $\Gamma$ up to conjugacy in $\bbH \rtimes \bbC^{\ast}$.

\begin{prop}
\label{prop:HB group of order 4}
    Let $\Gamma$ be a Heisenberg Bieberbach group satisfying $\proj_{2}(\Gamma) = \langle i \rangle$ and set $c = 4 / N$.
    Then $N$ must be even and $\Gamma$ is conjugate in $\bbH \rtimes \bbC^{\ast}$ to either
    \begin{equation}
        \Gamma_{c} \coloneqq \langle \Lambda_{i, c}, \psi \coloneqq (0, c / 4) \mu(i) \rangle \quad \text{ or } \quad \Gamma_{c}^{\prime} \coloneqq \langle \Lambda_{i, c}, \psi^{\prime} \coloneqq (0, 3 c / 4) \mu(i) \rangle.
    \end{equation}
\end{prop}

\begin{proof}
    Without loss of generality, we may assume that $\Gamma \cap \bbH = \Lambda_{\tau, c}$.
    In addition, we can assume that $\tau = i$ since $L_{\tau}$ is invariant under the multiplication by $i$.
    Take $g = (\zeta, \sigma) \in \bbH$ such that $\gamma \coloneqq g \mu(i) \in \Gamma$.
    Note that $\Gamma = \langle \Lambda_{i, c}, \gamma \rangle$.
    We can see that
    \begin{equation}
        \gamma^{2} = (\zeta, \sigma) (i \zeta, \sigma) \mu(i^{2}) = (\zeta + i \zeta, 2 \sigma - 2 \abs{\zeta}^{2}) \mu(- 1)
    \end{equation}
    and $\gamma^{4} = (0, 4 \sigma - 4 \abs{\zeta}^{2}) \in \Gamma \cap \bbH = \Lambda_{i, c}$,
    which implies $q \coloneqq (4 \sigma - 4 \abs{\zeta}^{2}) / c \in \bbZ$.
    Moreover, it follows from $\gamma \Lambda_{i, c} \gamma^{-1} = \Lambda_{i, c}$ that
    \begin{align}
        \gamma (1, 0) \gamma^{- 1} &= (i, - 4 \Real \bar{\zeta}) = (i, 0)(0, - 4 \Real \bar{\zeta}) \in \Lambda_{i, c}, \\
        \gamma (i, 0) \gamma^{- 1} &= (- 1, 4 \Imaginary \bar{\zeta}) = (- 1, 0) (0, 4 \Imaginary \bar{\zeta}) \in \Lambda_{i, c},
    \end{align}
    which implies $4 \Real \bar{\zeta}, 4 \Imaginary \bar{\zeta} \in c \bbZ$.
    Hence there exist $k, l \in \bbZ$ such that
    \begin{equation}
        \zeta = \frac{c}{4} (k + l i) = \frac{k + l i}{N}.
    \end{equation}
    It remains to determine the condition under which $\Gamma$ is torsion-free.
    If an element $\gamma'$ in $\Gamma \cap \proj_{2}^{-1}(i)$ or $\Gamma \cap \proj_{2}^{-1}(- i)$ is a torsion element, then so is $(\gamma')^{2} \in \Gamma \cap \proj_{2}^{-1}(- 1)$.
    Hence $\Gamma$ is torsion-free if and only if $\Gamma \cap \proj_{2}^{-1}(\langle - 1 \rangle) = \langle \Lambda_{i, c}, \gamma^{2} \rangle$ is torsion-free.
    It follows from the proof of Proposition~\ref{prop:HB group of order 2} that $\langle \Lambda_{i, c}, \gamma^{2} \rangle$ is torsion-free if and only if $N$ and $k - l$ are even and $q$ is odd.
    
    Now we set
    \begin{equation}
        \zeta' \coloneqq - \frac{\zeta}{1 - i} = - \frac{k - l + (k + l) i}{2 N}.
    \end{equation}
    It follows from Proposition~\ref{prop:conjugate} that $h \coloneqq (\zeta', 0) \in \bbH$ satisfies $h \Lambda_{i, c} h^{-1} = \Lambda_{i, c}$.
    Moreover, we have
    \begin{equation}
        h \gamma h^{-1} = (\zeta', 0) (\zeta, \sigma) (- i \zeta', 0) \mu(i) = (0, \sigma - \abs{\zeta}^{2}) \mu(i) = (0, q c / 4) \mu(i),
    \end{equation}
    and so $h \Gamma h^{-1} = \langle \Lambda_{i, c}, (0, q c / 4) \mu(i) \rangle$.
    Since $q$ is odd,  $\Gamma$ is conjugate in $\bbH \rtimes \bbC^{\ast}$ to
    \begin{equation}
        \langle \Lambda_{i, c}, (0, q c / 4) \mu(i) \rangle
        =
        \begin{cases}
            \Gamma_{c} & \text{if $q \equiv 1$}, \\
            \Gamma_{c}^{\prime} & \text{if $q \equiv 3$},
        \end{cases}
    \end{equation}
    which completes the proof.
\end{proof}

\begin{rem}
    Assume that $N \equiv 2$.
    Set $\zeta \coloneqq (1 + i) / 2$.
    It follows from Proposition~\ref{prop:conjugate} that $h \coloneqq (\zeta, 0) \in \bbH$ satisfies $h \Lambda_{i, c} h^{-1} = \Lambda_{i, c}$.
    Moreover, the equality $c N / 4 = 1$ implies
    \begin{equation}
        h (0, c / 4) \mu(i) h^{-1} = (\zeta, 0) (- i \zeta, c / 4) \mu(i) = (1, c (N - 2) / 4) (0, 3 c / 4) \mu(i).
    \end{equation}
    Since $(1, c (N - 2) / 4) \in \Lambda_{i, c}$,
    we have $h \Gamma_{c} h^{-1} = \Gamma_{c}^{\prime}$.
    On the other hand, if $N \equiv 0$,
    then $\Gamma_{c}$ is not conjugate to $\Gamma_{c}^{\prime}$ in $\bbH \rtimes \bbC^{\ast}$; see Remark~\ref{rem:non conjugate of order 4}.
\end{rem}

Proposition~\ref{prop:HB group of order 4} implies that $\Gamma \backslash \bbH$ is isomorphic to either $(M \coloneqq \Gamma_{c} \backslash \bbH, T^{1, 0} M, r \theta_{\bbH})$ or $(M^{\prime} \coloneqq \Gamma_{c}^{\prime} \backslash \bbH, T^{1, 0} M^{\prime}, r \theta_{\bbH})$ for some $r > 0$ as a pseudo-Hermitian manifold.
Hence it suffices to study the joint spectrum of $\Delta_{b}$ and $i^{- 1} T$ on these pseudo-Hermitian manifolds.
We first consider $(M, T^{1, 0} M, r \theta_{\bbH})$.

\begin{thm}
    The multiplicity $m_{n, k}$ of $\sigma_{n, k}$ in the joint spectrum of $\Delta_{b}$ and $i^{- 1} T$ on $M$ is given by
    \begin{equation}
        m_{n, k} =
        \begin{cases}
            \frac{1}{4} \left( \abs{n} N + 2 (- 1)^{n + k} + 2 \cos \frac{\pi (\abs{n} - k)}{2} - 2 \sin \frac{\pi (\abs{n} - k)}{2} \right) & \text{if $\abs{n} N \equiv 0$}, \\
            \frac{1}{4} (\abs{n} N + 2 (- 1)^{n + k}) & \text{if $\abs{n} N \equiv 2$}.
        \end{cases}
    \end{equation}
    The multiplicity of $(\kappa, 0)$ in the joint spectrum is equal to $\# E_{\kappa} / 4$ for any $\kappa > 0$.
\end{thm}

\begin{proof}
    Any function on $M$ can be identified with a $\psi$-invariant function on $M_{i, c}$.
    Denote by $\mathscr{H}_{n}^{\psi}$ the space of $\psi$-invariant functions in $\mathscr{H}_{n}$.
    Then $L^{2}(M) = \bigoplus_{n \in\bbZ} \mathscr{H}_{n}^{\psi}$.
    
    In the case $n = 0$, one has $\mathscr{H}_{0} = \bigoplus_{\zeta \in L_{i}^{\prime}} \bbC \chi_{\zeta}$,
    and $\Delta_{b} \chi_{\zeta} = (\pi^{2} \abs{\zeta}^{2} / r) \chi_{\zeta}$.
    Since $\psi^{\ast} \chi_{\zeta} = \chi_{- i \zeta}$,
    the multiplicity of the eigenvalue $(\kappa, 0)$ in the joint spectrum of $\Delta_{b}$ and $i^{- 1} T$ on $M$ is equal to $\# E_{\kappa} / 4$ for any $\kappa > 0$.

    Assume that $n > 0$.
    Set $\mathscr{H}_{n, k}^{\psi} \coloneqq \mathscr{H}_{n, k} \cap \mathscr{H}_{n}^{\psi}$, which is the space of $\psi$-invariant functions in $\mathscr{H}_{n, k}$.
    Note that $\mathscr{H}_{n}^{\psi} = \bigoplus_{k \in \bbZ_{\geq 0}} \mathscr{H}_{n, k}^{\psi}$ and
    $\mathscr{H}_{n, k}^{\psi}$ is the joint eigenspace corresponding to the eigenvalue $\sigma_{n, k}$ of $\Delta_{b}$ and $i^{- 1} T$ on $M$.
    To compute the dimension of $\mathscr{H}_{n, k}^{\psi}$,
    we consider the representation $\rho$ of $\Gamma_{c} / \Lambda_{i, c} \cong \langle \psi \rangle \cong \bbZ / 4 \bbZ$ on $\mathscr{H}_{n, k}$ defined by the pullback.
    The Schur orthogonality relation implies that
    \begin{equation}
        \dim \mathscr{H}_{n, k}^{\psi} = \frac{1}{4} \sum_{l = 0}^{3} \chi_{\rho}(\psi^{l}) = \frac{1}{4} (n N + 2 \Real \chi_{\rho}(\psi) + \chi_{\rho}(\psi^{2})).
    \end{equation}
    As we showed in the proof of Theorem~\ref{thm:joint spectrum of order 2}, $\chi_{\rho}(\psi^{2}) = \chi_{\rho}(\varphi) = 2 (- 1)^{n + k}$, where $\varphi$ is as introduced in Proposition~\ref{prop:HB group of order 2}.
    Hence it suffices to compute the value $\chi_{\rho}(\psi)$.
    
    For each $0 \leq j \leq n N - 1$, it follows from $N = 4 / c$ that
    \begin{align}
        (\psi^{\ast} f_{j})(z, t)
        &= f_{j}(i z, t + c / 4) \\
        &= \exp \left( \frac{2 \pi n}{c} \left(i t + \frac{i c}{4} - \abs{i z}^{2} + (i z)^{2} \right) \right) h_{j}(i z) \\
        &= i^{n} \exp \left( \frac{2 \pi n}{c} \left(i t - \abs{z}^{2} + z^{2} -2 z^2 \right) \right) h_{j}(i z) \\
        &= i^{n} \exp \left( \frac{2 \pi n}{c} (i t - \abs{z}^{2} + z^{2}) \right) \exp(- \pi n N z^{2}) h_{j}(i z).
    \end{align}
    Now we get
        \begin{align}
            \exp(- \pi n N z^{2}) h_{j}(i z)
            & = \sum_{m \in \bbZ} \exp \left( - \pi n N z^{2} -\pi n N \left( m + \frac{j}{n N} \right)^{2} - 2 \pi  n N \left( m + \frac{j}{n N} \right) z \right) \\
            & = \sum_{m \in \bbZ} \exp \left( - \pi n N \left( m + \frac{j}{n N} + z \right)^{2} \right) \label{exp-h_j(iz)}.
        \end{align}
    Consider $g(x) \coloneqq \exp(- \pi n N (x + j / n N + z)^{2}) \in \mathcal{S}(\bbR)$.
    Then, the equation \eqref{exp-h_j(iz)} and the Poisson summation formula give
    \begin{equation}
        \exp(- \pi n N z^{2}) h_{j}(i z) = \sum_{m \in \bbZ} g(m) = \sum_{l = 0}^{n N - 1} \sum_{m^{\prime} \in \bbZ} \hat{g}(n N m^{\prime} + l),
    \end{equation}
    where $\hat{g}$ is the Fourier transform of $g$.
    A computation yields
    \begin{align}
        \hat{g}(\xi)
        &= \int_{\bbR} \exp \left( - \pi n N \left( x + \frac{j}{n N} + z \right)^{2} - 2 \pi i x \xi \right) \, d x \\
        &= \frac{1}{\sqrt{n N}} \exp \left( - \frac{\pi}{n N} \xi^{2} + 2 \pi i \left( \frac{j}{n N} + z \right) \xi \right),
    \end{align}
    which implies
    \begin{align}
        &\exp(- \pi N z^{2}) h_{j}(i z) \\
        &= \frac{1}{\sqrt{n N}} \sum_{l = 0}^{n N - 1} \sum_{m^{\prime} \in \bbZ} \exp \left( - \frac{\pi}{n N} (n N m^{\prime} + l)^{2} + 2 \pi i \left( \frac{j}{n N} + z \right) (n N m^{\prime} + l) \right) \\
        &= \frac{1}{\sqrt{n N}} \sum_{l = 0}^{n N - 1} \exp \left( \frac{2 \pi i}{n N}j l  \right) h_{l}(z).
    \end{align}
    Hence we have
    \begin{align}
        (\psi^{\ast} f_{j})(z, t)
        &= i^{n} \exp \left( \frac{2 \pi n}{c} (i t - \abs{z}^{2} + z^{2}) \right) \frac{1}{\sqrt{n N}}  \sum_{l = 0}^{n N - 1} \exp \left( \frac{2 \pi i}{n N} j l  \right) h_{l}(z) \\
        &= \frac{i^{n}}{\sqrt{n N}} \sum_{l = 0}^{n N - 1} \exp \left( \frac{2 \pi i}{n N} j l
        \right) f_{l}(z).
    \end{align}
    Moreover, one can see that
    \begin{equation}
        (\psi^{\ast} (Z f))(z, t) = (Z f)(i z, t + c / 4) = - i (Z (\psi^{\ast} f))(z, t)
    \end{equation}
    for any $f \in C^{\infty}(\bbH)$.
    Thus we have
    \begin{equation}
        \psi^{\ast} (Z^{k} f_{j}) = (- i)^{k} Z^{k} (\psi^{\ast} f_j) = \frac{i^{n - k}}{\sqrt{n N}} \sum_{l = 0}^{n N - 1} \exp \left( \frac{2 \pi i}{n N} j l \right) Z^{k} f_{l}(z).
    \end{equation}
    This formula gives that
    \begin{equation}
        \chi_{\rho}(\psi) = \frac{i^{n - k}}{\sqrt{n N}} \sum_{j = 0}^{n N - 1} \exp \left( \frac{2 \pi i}{n N} j^{2} \right) =
        \begin{cases}
            i^{n - k} (1 + i) & \text{if $n N \equiv 0$}, \\
            0 & \text{if $n N \equiv 2$}.
        \end{cases}
    \end{equation}
    Here, we use a result on the quadratic Gauss sum; see \cite{Lang1994ANT}*{Chapter IV.3} for example.
    Remark that, since $N$ is even, $n N \equiv 0$ or $2$.
    Note that
    \begin{equation}
        \Real (i^{n - k} (1 + i)) = \Real (e^{\pi (n - k) i / 2} (1 + i)) = \cos \frac{\pi (n - k)}{2} - \sin \frac{\pi (n - k)}{2}.
    \end{equation}
    Therefore we have 
    \begin{equation}
        \dim \mathscr{H}_{n, k}^{\psi} = 
        \begin{cases}
            \frac{1}{4} \left( n N + 2 (- 1)^{n + k} + 2 \cos \frac{\pi (n - k)}{2} - 2 \sin \frac{\pi (n - k)}{2} \right) & \text{if $n N \equiv 0$}, \\
            \frac{1}{4} (n N + 2 (- 1)^{n + k}) & \text{if $n N \equiv 2$}.
        \end{cases}
    \end{equation}

    In the case $n < 0$, we have the orthogonal decomposition $\mathscr{H}_{n}^{\psi} = \bigoplus_{k \geq 0} \overline{\mathscr{H}_{- n, k}^{\psi}}$ and $\overline{\mathscr{H}_{- n, k}^{\psi}}$ is the joint eigenspace corresponding to the eigenvalue $\sigma_{n, k}$ of $\Delta_{b}$ and $i^{- 1} T$ on $M$.
    In particular,
    \begin{align}
        \dim \overline{\mathscr{H}_{- n, k}^{\psi}} &= \dim \mathscr{H}_{- n, k}^{\psi} \\
        &=
        \begin{cases}
            \frac{1}{4} \left( \abs{n} N + 2 (- 1)^{n + k} + 2 \cos \frac{\pi (\abs{n} - k)}{2} - 2 \sin \frac{\pi (\abs{n} - k)}{2} \right) & \text{if $\abs{n} N \equiv 0$}, \\
            \frac{1}{4} (\abs{n} N + 2 (- 1)^{n + k}) & \text{if $\abs{n} N \equiv 2$},
        \end{cases}
    \end{align}
    which completes the proof.
\end{proof}

We next turn to the case of $\Gamma_{c}^{\prime}$.
Since $\psi^{\prime} = (0, c / 2) \psi$, we can use the computation in the case of $\Gamma_{c}$.
Note that this case was already considered by the first author~\cite{Suzuki2024preprint}*{Section 5}.

\begin{thm}
    The multiplicity $m_{n, k}^{\prime}$ of $\sigma_{n, k}$ in the joint spectrum of $\Delta_{b}$ and $i^{- 1} T$ on $M'$ is given by
    \begin{equation}
        m_{n, k}^{\prime} =
        \begin{cases}
            \frac{1}{4} \left( \abs{n} N + 2 (- 1)^{n + k} + 2 \cos \frac{\pi (\abs{n} + k)}{2} + 2 \sin \frac{\pi (\abs{n} + k)}{2} \right) & \text{if $\abs{n} N \equiv 0$}, \\
            \frac{1}{4} (\abs{n} N + 2 (- 1)^{n + k}) & \text{if $\abs{n} N \equiv 2$}.
        \end{cases}
    \end{equation}
    The multiplicity of $(\kappa, 0)$ in the joint spectrum is equal to $\# E_{\kappa} / 4$ for any $\kappa > 0$.
\end{thm}

\begin{proof}
    Any function on $M^{\prime}$ can be identified with a $\psi^{\prime}$-invariant function on $M_{i, c}$.
    Denote by $\mathscr{H}_{n}^{\psi^{\prime}}$ the space of $\psi^{\prime}$-invariant functions in $\mathscr{H}_{n}$.
    Then $L^{2}(M^{\prime}) = \bigoplus_{n \in\bbZ} \mathscr{H}_{n}^{\psi^{\prime}}$.
    The proof for $n = 0$ is the same as that in the case of $M$.
    
    Assume that $n > 0$.
    Set $\mathscr{H}_{n, k}^{\psi^{\prime}} \coloneqq \mathscr{H}_{n, k} \cap \mathscr{H}_{n}^{\psi^{\prime}}$, which is the space of $\psi^{\prime}$-invariant functions in $\mathscr{H}_{n, k}$.
    Note that $\mathscr{H}_{n}^{\psi^{\prime}} = \bigoplus_{k \in \bbZ_{\geq 0}} \mathscr{H}_{n, k}^{\psi^{\prime}}$ and
    $\mathscr{H}_{n, k}^{\psi^{\prime}}$ is the joint eigenspace corresponding to the eigenvalue $\sigma_{n, k}$ of $\Delta_{b}$ and $i^{- 1} T$ on $M^{\prime}$.
    To compute the dimension of $\mathscr{H}_{n, k}^{\psi^{\prime}}$,
    we consider the representation $\rho^{\prime}$ of $\Gamma_{c}^{\prime} / \Lambda_{i, c} \cong \langle \psi^{\prime} \rangle \cong \bbZ / 4 \bbZ$ on $\mathscr{H}_{n, k}$ defined by the pullback.
    The Schur orthogonality relation implies that
    \begin{equation}
        \dim \mathscr{H}_{n, k}^{\psi^{\prime}} = \frac{1}{4} \sum_{l = 0}^{3} \chi_{\rho^{\prime}}((\psi^{\prime})^{l}) = \frac{1}{4} (n N + 2 \Real \chi_{\rho^{\prime}}(\psi^{\prime}) + \chi_{\rho^{\prime}}((\psi^{\prime})^{2}).
    \end{equation}
    Since $\psi^{\prime} = (0, c / 2) \psi$,
    \begin{equation}
        ((\psi^{\prime})^{\ast} f)(z, t) = f((0, c / 2) \psi(z, t)) = (- 1)^{n} (\psi^{\ast} f)(z, t)
    \end{equation}
    for any $f \in \mathscr{H}_{n}$.
    Thus we have
    \begin{equation}
        \chi_{\rho^{\prime}}(\psi^{\prime}) = (- 1)^{n} \chi_{\rho}(\psi) =
        \begin{cases}
            (- i)^{n + k} (1 + i) & \text{if $n N \equiv 0$}, \\
            0 & \text{if $n N \equiv 2$},
        \end{cases}
    \end{equation}
    and $\chi_{\rho^{\prime}}((\psi^{\prime})^{2}) = \chi_{\rho}(\psi^{2}) = 2 (- 1)^{n + k}$.
    Note that
    \begin{equation}
        \Real ((- i)^{n + k} (1 + i)) = \Real (e^{- \pi (n + k) i / 2} (1 + i)) = \cos \frac{\pi (n + k)}{2} + \sin \frac{\pi (n + k)}{2}.
    \end{equation}
    Therefore we have 
    \begin{equation}
        \dim \mathscr{H}_{n, k}^{\psi^{\prime}} = 
        \begin{cases}
            \frac{1}{4} \left( n N + 2 (- 1)^{n + k} + 2 \cos \frac{\pi (n + k)}{2} + 2 \sin \frac{\pi (n + k)}{2} \right) & \text{if $n N \equiv 0$}, \\
            \frac{1}{4} (n N + 2 (- 1)^{n + k}) & \text{if $n N \equiv 2$}.
        \end{cases}
    \end{equation}

    In the case $n < 0$, we have the orthogonal decomposition $\mathscr{H}_{n}^{\psi^{\prime}} = \bigoplus_{k \geq 0} \overline{\mathscr{H}_{- n, k}^{\psi^{\prime}}}$ and $\overline{\mathscr{H}_{- n, k}^{\psi^{\prime}}}$ is the joint eigenspace corresponding to the eigenvalue $\sigma_{n, k}$ of $\Delta_{b}$ and $i^{- 1} T$ on $M^{\prime}$.
    In particular,
    \begin{align}
        \dim \overline{\mathscr{H}_{- n, k}^{\psi^{\prime}}} &= \dim \mathscr{H}_{- n, k}^{\psi^{\prime}} \\
        &=
        \begin{cases}
            \frac{1}{4} \left( \abs{n} N + 2 (- 1)^{n + k} + 2 \cos \frac{\pi (\abs{n} + k)}{2} + 2 \sin \frac{\pi (\abs{n} + k)}{2} \right) & \text{if $\abs{n} N \equiv 0$}, \\
            \frac{1}{4} (\abs{n} N + 2 (- 1)^{n + k}) & \text{if $\abs{n} N \equiv 2$},
        \end{cases}
    \end{align}
    which completes the proof.
\end{proof}

\begin{rem}
\label{rem:non conjugate of order 4}
    Consider the case $N \equiv 0$.
    Suppose to the contrary that $\Gamma_{c}$ is conjugate in $\bbH \rtimes \bbC^{\ast}$ to $\Gamma_{c}^{\prime}$.
    Then $(M, T^{1, 0} M, r \theta_{\bbH})$ is isomorphic to $(M^{\prime}, T^{1, 0} M^{\prime}, r^{\prime} \theta_{\bbH})$ for some $r^{\prime} \in \bbR_{> 0}$ as a pseudo-Hermitian manifold.
    Considering the spectrum of $i^{- 1} T$ yields that $r = r^{\prime}$,
    and so $m_{n, k}$ must be equal to $m_{n, k}^{\prime}$ for every $n$ and $k$.
    However, we can see that
    \begin{equation}
        m_{1, 1} = \frac{N}{4} + 1 \neq \frac{N}{4} = m_{1, 1}^{\prime},
    \end{equation}
    which is a contradiction.
    Therefore $\Gamma_{c}$ is not conjugate to $\Gamma_{c}^{\prime}$ in $\bbH \rtimes \bbC^{\ast}$.
\end{rem}

\section{Joint spectrum on Heisenberg Bieberbach manifolds $\Gamma \backslash \bbH$ with $\proj_{2}(\Gamma) = \langle e^{2 \pi i / 3} \rangle$}

To simplify notation, we write $\omega$ for $e^{2 \pi i / 3}$.
In this section, $\equiv$ is understood modulo $3$.
Let $\Gamma$ be a Heisenberg Bieberbach group with $\proj_{2}(\Gamma) = \langle \omega \rangle$ and consider the Heisenberg Bieberbach manifold $\Gamma \backslash \bbH$.
Recall that $N = \# (Z(\Lambda) / [\Lambda, \Lambda])$, where $\Lambda = \Gamma \cap \bbH$.
We first give a classification of such $\Gamma$ up to conjugacy in $\bbH \rtimes \bbC^{\ast}$.

\begin{prop}
\label{prop:HB group of order 3}
    Let $\Gamma$ be a Heisenberg Bieberbach group satisfying $\proj_{2}(\Gamma) = \langle \omega \rangle$ and set $c = 4 \Imaginary \omega / N$.
    The conjugacy class of $\Gamma$ in $\bbH \rtimes \bbC^{\ast}$ depends on $N$ modulo $3$ as follows.
    If $N \equiv 0$, then it is conjugate to either
    \begin{equation}
        \Gamma_{c} \coloneqq \langle \Lambda_{\omega, c}, \varphi \coloneqq (\omega / 2, c / 3 + \Imaginary \omega / 6) \mu(\omega) \rangle
    \end{equation}
    or
    \begin{equation}
         \Gamma_{c}^{\prime} \coloneqq \langle \Lambda_{\omega, c}, \varphi^{\prime} \coloneqq (\omega / 2, 2 c / 3 + \Imaginary \omega / 6) \mu(\omega) \rangle.
    \end{equation}
    If $N \equiv 1$, then it is conjugate to $\Gamma_{c}^{\prime}$.
    If $N \equiv 2$, then it is conjugate to $\Gamma_{c}$.
\end{prop}

\begin{proof}
    Without loss of generality, we may assume that $\Gamma \cap \bbH = \Lambda_{\tau, c}$.
    In addition, we can assume that $\tau = \omega$ since $L_{\tau}$ is invariant under the multiplication by $\omega$.
    Take $g = (\zeta, \sigma) \in \bbH$ such that $\gamma \coloneqq g \mu(\omega) \in \Gamma$.
    Note that $\Gamma = \langle \Lambda_{\omega, c}, \gamma \rangle$.
    We can see that $q \coloneqq (3 \sigma - 2  \abs{\zeta}^{2} \Imaginary \omega) / c \in \bbZ$ since
    \begin{equation}
        \gamma^{3} = (\zeta, \sigma) (\omega \zeta, \sigma) (\omega^{2} \zeta, \sigma) \mu(\omega^{3}) = (0, 3 \sigma - 2 \abs{\zeta}^{2} \Imaginary \omega) \in \Gamma \cap \bbH = \Lambda_{\omega, c}.
    \end{equation}
    Moreover, it follows from $\gamma \Lambda_{\omega, c} \gamma^{-1} = \Lambda_{\omega, c}$ that
    \begin{align}
        \gamma (1, 0) \gamma^{- 1} &= (\omega, - 4 \Imaginary (\omega \bar{\zeta})) = (\omega, 0) (0, - 4 \Imaginary (\omega \bar{\zeta})) \in \Lambda_{\omega, c}, \\
        \gamma (\omega, 0) \gamma^{- 1} &= (\omega^{2}, - 4 \Imaginary (\omega^{2} \bar{\zeta})) = (- 1, 0) (- \omega, 0) (0, - 4 \Imaginary (\omega^{2} \bar{\zeta}) + 2 \Imaginary \omega) \in \Lambda_{\omega, c},
    \end{align}
    which implies $4 \Imaginary (\omega \bar{\zeta}), 4 \Imaginary (\omega^{2} \bar{\zeta}) - 2 \Imaginary \omega \in c \bbZ$.
    Hence there exist $k, l \in \bbZ$ such that
    \begin{equation}
        \zeta = \frac{c}{4 \Imaginary \omega} (k + l \omega) + \frac{1}{2} \omega = \frac{k + l \omega}{N} + \frac{1}{2} \omega.
    \end{equation}
    If $N \not\equiv 0$, then by translation we may assume that $k + l \equiv 0$.
    It remains to determine the condition under which $\Gamma$ is torsion-free.
    Any element of $\Lambda_{\omega, c} \setminus \{(0, 0)\}$ has infinite order.
    Moreover, if an element $\gamma'$ in $\Gamma \cap \proj_{2}^{-1}(\omega^{2})$ is a torsion element, then so is $(\gamma')^{2} \in \Gamma \cap \proj_{2}^{-1}(\omega)$.
    Hence it suffices to check the condition that every element in $\Gamma \cap \proj_{2}^{-1}(\omega)$ is of infinite order.
    Any element $\gamma' \in \Gamma \cap \proj_{2}^{-1}(\omega)$ is of the form
    \begin{equation}
        \gamma' = (m + n \omega, 2 m n \Imaginary \omega + p c) \gamma = (\zeta + m + n \omega, \sigma + m n N c / 2 + 2 \Imaginary ((m + n \omega) \bar{\zeta}) + p c) \mu(\omega)
    \end{equation}
    for some $m, n, p \in \bbZ$.
    This  $\gamma'$ is not torsion if and only if
    \begin{equation}
        (\gamma')^{3} = (0, 3 \sigma + 3 m n N c / 2 + 6 \Imaginary ((m + n \omega) \bar{\zeta}) + 3 p c - 2 \abs{\zeta + m + n \omega}^{2} \Imaginary \omega) \neq (0, 0).
    \end{equation}
    Since $p \in \bbZ$ is arbitrary,
    $\Gamma$ is torsion-free if and only if
    \begin{align}
        &\frac{3 \sigma}{c} + \frac{3 m n N}{2} + \frac{6}{c} \Imaginary ((m + n \omega) \bar{\zeta}) - \frac{2 \Imaginary \omega}{c} \abs{\zeta + m + n \omega}^{2} \\
        &=\frac{3 \sigma}{c} + \frac{3 m n N}{2} + \frac{6}{c} \Imaginary ((m + n \omega) \bar{\zeta}) - \frac{2 \Imaginary \omega}{c} 
            (\abs{\zeta}^2 + 2 \Real ((m + n \omega) \bar{\zeta}) + \abs{m + n \omega}^2
            )\\
        &= q + \frac{3 m n N}{2} + \frac{6}{c} \Imaginary ((m + n \omega) \bar{\zeta}) - N \Real ((m + n \omega) \bar{\zeta}) - \frac{N}{2} \abs{m + n \omega}^{2}
    \end{align}
    is not congruent to $0$ modulo $3$ for any $m, n \in \bbZ$.
    A computation yields
    \begin{gather}
        \Real ((m + n \omega) \bar{\zeta}) = \frac{1}{N} \left(m k - \frac{1}{2} m l - \frac{1}{2} n k + n l \right) - \frac{1}{4} m + \frac{1}{2}{n}, \\
        \Imaginary ((m + n \omega) \bar{\zeta}) = \frac{c}{4} (- m l + n k) - \frac{c}{8} m N;
    \end{gather}
    here we use the fact that $c N = 4 \Imaginary \omega$.
    Hence
    \begin{align}
        &\frac{3 \sigma}{c} + \frac{3 m n N}{2} + \frac{6}{c} \Imaginary ((m + n \omega) \bar{\zeta}) - \frac{2 \Imaginary \omega}{c} \abs{\zeta + m + n \omega}^{2} \\
        &= q - m k - m l + 2 n k - n l + 2 m n N - \frac{1}{2} m (m + 1) N - \frac{1}{2} n (n + 1) N \\
        &\equiv q - (m + n) (k + l) - \frac{1}{2} (m + n) (m + n + 1) N \\
        &\equiv
        \begin{cases}
            q & \text{if $m + n \equiv 0$}, \\
            q - (k + l + N) & \text{if $m + n \equiv 1$}, \\
            q - 2 (k + l) & \text{if $m + n \equiv 2$}.
        \end{cases}
    \end{align}
    In the case $N \equiv 0$, this is not congruent to $0$ modulo $3$ only if $k + l \equiv 0$.
    Hence we may assume that $k + l \equiv 0$ for any $N$,
    and in this case, $\Gamma$ is torsion-free if and only if $q \not\equiv 0, N$.
    Set
    \begin{equation}
        \zeta' \coloneqq - \frac{\zeta - \omega / 2}{1 - \omega} = - \frac{2 k - l + (k + l) \omega}{3 N}.
    \end{equation}
    It follows from Proposition~\ref{prop:conjugate} and $k + l \equiv 0$ that $h \coloneqq (\zeta', 0) \in \bbH$ satisfies $h \Lambda_{\omega, c} h^{-1} = \Lambda_{\omega, c}$.
    Moreover, we have
    \begin{equation}
        h \gamma h^{-1} = (\zeta', 0) (\zeta, \sigma) (- \omega \zeta', 0) \mu(\omega) = (\omega / 2 , q c / 3 +  \Imaginary \omega / 6) \mu(\omega).
    \end{equation}
    Hence $\Gamma$ is conjugate in $\bbH \rtimes \bbC^{\ast}$ to
    \begin{equation}
        \langle \Lambda_{\omega, c}, (\omega / 2 , q c / 3 +  \Imaginary \omega / 6) \mu(\omega) \rangle
        =
        \begin{cases}
            \Gamma_{c} & \text{if $q \equiv 1$}, \\
            \Gamma_{c}^{\prime} & \text{if $q \equiv 2$},
        \end{cases}
    \end{equation}
    which completes the proof.
\end{proof}

This proposition implies that $\Gamma \backslash \bbH$ is isomorphic to either $(M \coloneqq \Gamma_{c} \backslash \bbH, T^{1, 0} M, r \theta_{\bbH})$ or $(M^{\prime} \coloneqq \Gamma_{c}^{\prime} \backslash \bbH, T^{1, 0} M^{\prime}, r \theta_{\bbH})$ for some $r > 0$ as a pseudo-Hermitian manifold.
Hence it suffices to study the joint spectrum of $\Delta_{b}$ and $i^{- 1} T$ on these pseudo-Hermitian manifolds.
We first consider $(M, T^{1, 0} M, r \theta_{\bbH})$.

\begin{thm}
    The multiplicity $m_{n, k}$ of $\sigma_{n, k}$ in the joint spectrum of $\Delta_{b}$ and $i^{- 1} T$ on $M$ is given by
    \begin{equation}
        m_{n, k} = \frac{1}{3} \left( \abs{n} N + \frac{4}{\sqrt{3}} \cos \left( \frac{2 \pi (\abs{n} - k)}{3} + \frac{\pi}{6} \right) + \frac{2}{\sqrt{3}} \cos \left( \frac{2 \pi (\abs{n} - k - \abs{n} N)}{3} + \frac{\pi}{6} \right) \right).
    \end{equation}
    The multiplicity of $(\kappa, 0)$ in the joint spectrum is equal to $\# E_{\kappa} / 3$ for any $\kappa > 0$.
\end{thm}

\begin{proof}
    Any function on $M$ can be identified with a $\varphi$-invariant function on $M_{\omega, c}$.
    Denote by $\mathscr{H}_{n}^{\varphi}$ the space of $\varphi$-invariant functions in $\mathscr{H}_{n}$.
    Then $L^{2}(M) = \bigoplus_{n \in\bbZ} \mathscr{H}_{n}^{\varphi}$.
    
    In the case $n = 0$, one has $\mathscr{H}_{0} = \bigoplus_{\zeta \in L_{\omega}^{\prime}} \bbC \chi_{\zeta}$
    and $\Delta_{b} \chi_{\zeta} = (\pi^{2} \abs{\zeta}^{2} / r) \chi_{\zeta}$.
    Since $\varphi^{\ast} \chi_{\zeta} = (- 1)^{\Real (\zeta \overline{\omega})} \chi_{\overline{\omega} \zeta}$,
    the multiplicity of the eigenvalue $(\kappa, 0)$ in the joint spectrum of $\Delta_{b}$ and $i^{- 1} T$ on $M$ is equal to $\# E_{\kappa} / 3$ for any $\kappa > 0$.

    Assume that $n > 0$.
    Set $\mathscr{H}_{n, k}^{\varphi} \coloneqq \mathscr{H}_{n, k} \cap \mathscr{H}_{n}^{\varphi}$, which is the space of $\varphi$-invariant functions in $\mathscr{H}_{n, k}$.
    Note that $\mathscr{H}_{n}^{\varphi} = \bigoplus_{k \in \bbZ_{\geq 0}} \mathscr{H}_{n, k}^{\varphi}$ and
    $\mathscr{H}_{n, k}^{\varphi}$ is the joint eigenspace corresponding to the eigenvalue $\sigma_{n, k}$ of $\Delta_{b}$ and $i^{- 1} T$ on $M$.
    To compute the dimension of $\mathscr{H}_{n, k}^{\varphi}$,
    we consider the representation $\rho$ of $\Gamma_{c} / \Lambda_{\omega, c} \cong \langle \varphi \rangle \cong \bbZ / 3 \bbZ$ on $\mathscr{H}_{n, k}$ defined by the pullback.
    The Schur orthogonality relation implies that
    \begin{equation}
        \dim \mathscr{H}_{n, k}^{\varphi} = \frac{1}{3} \sum_{l = 0}^{2} \chi_{\rho}(\varphi^{l}) = \frac{1}{3} (n N + 2 \Real \chi_{\rho}(\varphi)).
    \end{equation}
    Hence it suffices to compute the value $\chi_{\rho}(\varphi)$.
    
    For each $0 \leq j \leq n N - 1$, it follows from $N = 4 \Imaginary \omega / c$ that
    \begin{align}
        &(\varphi^{\ast} f_{j})(z, t) \\
        &= f_{j} ((\omega / 2, c / 3 + \Imaginary \omega / 6) (\omega z, t)) \\
        &= f_{j}(\omega z + \omega / 2, t + c / 3 + \Imaginary \omega / 6 - \Imaginary z) \\
        &= \exp \left( \frac{2 \pi n}{c} \left(i t + \frac{i c}{3} + \frac{i \Imaginary \omega}{6} - i \Imaginary z - \abs{\omega z + \omega / 2}^{2} + (\omega z + \omega / 2)^{2} \right) \right) h_{j}(\omega z + \omega / 2) \\
        &= e^{2 \pi i n / 3 + \pi i n N / 12} 
        \exp \left( \frac{2 \pi n}{c} \left(i t - \abs{z}^2 -z - \frac{1}{4} + \omega^{2} \left( z + \frac{1}{2} \right)^{2}
        \right) 
            \right) h_{j}(\omega z + \omega / 2) \\
        &= e^{2 \pi i n / 3 + \pi i n N / 12} 
        \exp \left( \frac{2 \pi n}{c} \left(i t - \abs{z}^2 + z^2 + \omega (\omega - \overline{\omega}) \left( z + \frac{1}{2} \right)^2 \right) \right) h_{j}(\omega z + \omega / 2) \\
        &= e^{2 \pi i n / 3 + \pi i n N / 12} \exp \left( \frac{2 \pi n}{c} (i t - \abs{z}^{2} + z^{2}) \right) \exp \left( \pi i n N \omega \left( z + \frac{1}{2} \right)^{2} \right) h_{j}(\omega z + \omega / 2).
    \end{align}
    Now we have
        \begin{equation}\label{h_j-omega}
            \exp \left( \pi i n N \omega \left( z + \frac{1}{2} \right)^{2} \right) h_{j}(\omega z + \omega / 2)
            = \sum_{m \in \bbZ} \exp \left(
                \pi i n N \omega \left( m + \frac{j}{nN} + z + \frac{1}{2} \right)^2
            \right).
        \end{equation}
    Consider $g(x) \coloneqq \exp(\pi i n N \omega (x + j / n N + z + 1 / 2)^{2}) \in \mathcal{S}(\bbR)$.
    Then the equation \eqref{h_j-omega} and the Poisson summation formula give
    \begin{equation}
        \exp \left( \pi i n N \omega \left( z + \frac{1}{2} \right)^{2} \right) h_{j}(\omega z + \omega / 2) = \sum_{m \in \bbZ} g(m) = \sum_{l = 0}^{n N - 1} \sum_{m^{\prime} \in \bbZ} \hat{g}(n N m^{\prime} + l),
    \end{equation}
    where $\hat{g}$ is the Fourier transform of $g$.
    A computation yields
    \begin{align}
        \hat{g}(\xi)
        &= \int_{\bbR} \exp \left( \pi i n N \omega \left( x + \frac{j}{n N} + z + \frac{1}{2} \right)^{2} - 2 \pi i x \xi \right) \, d x \\
        &= \frac{1}{\sqrt{n N}} e^{- \pi i / 12} \exp \left( - \frac{\pi i}{n N \omega} \xi^{2} + 2 \pi i \left( \frac{j}{n N} + z + \frac{1}{2} \right) \xi \right) \\
        &= \frac{1}{\sqrt{n N}} e^{- \pi i / 12} \exp \left( \frac{\pi i (\omega + 1)}{n N} \xi^{2} + 2 \pi i \left( \frac{j}{n N} + z + \frac{1}{2} \right) \xi \right),
    \end{align}
    which implies
    \begin{align}
        &\exp \left( \pi i n N \omega \left( z + \frac{1}{2} \right)^{2} \right) h_{j}(\omega z + \omega / 2) \\
        &= \frac{1}{\sqrt{n N}} e^{- \pi i / 12} \sum_{l = 0}^{n N - 1} \sum_{m^{\prime}} \exp \left( \frac{\pi i (\omega + 1)}{n N} (n N m^{\prime} + l)^{2} + 2 \pi i \left( \frac{j}{n N} + z + \frac{1}{2} \right) (n N m^{\prime} + l) \right) \\
        &= \frac{1}{\sqrt{n N}} e^{- \pi i / 12} \sum_{l = 0}^{n N - 1} \exp \left( \frac{\pi i}{n N} l^2 + \frac{2 \pi i}{n N} j l + \pi i l \right) h_{l}(z).
    \end{align}
    Hence we have
    \begin{align}
        (\varphi^{\ast} f_{j})(z, t)
        &= e^{2 \pi i n / 3 + \pi i n N / 12} \exp \left( \frac{2 \pi n}{c} (i t - \abs{z}^{2} + z^{2}) \right) \\
        &\quad \times \frac{1}{\sqrt{n N}} e^{- \pi i / 12} \sum_{l = 0}^{n N - 1} \exp \left( \frac{\pi i}{n N} j^2 + \frac{2 \pi i}{n N} j l + \pi i l \right) h_{l}(z) \\
        &= \frac{1}{\sqrt{n N}} e^{2 \pi i n / 3 + \pi i (nN - 1) / 12} \sum_{l = 0}^{n N - 1} \exp \left( \frac{\pi i}{n N} l^2 + \frac{2 \pi i}{n N} j l + \pi i l \right) f_{l}(z).
    \end{align}
    Moreover, one can see that
    \begin{equation}
        (\varphi^{\ast} (Z f))(z, t) = (Z f)(\omega z + \omega / 2, t + c / 3 + \Imaginary \omega / 6 - \Imaginary z) = e^{- 2 \pi i / 3} (Z (\varphi^{\ast} f))(z, t)
    \end{equation}
    for any $f \in C^{\infty}(\bbH)$.
    Thus we have
    \begin{align}
        \varphi^{\ast} (Z^{k} f_{j})
        &= e^{- 2 \pi i k / 3} Z^{k} (\varphi^{\ast} f_j) \\
        &= \frac{1}{\sqrt{n N}} e^{2 \pi i (n - k) / 3 + \pi i (nN - 1) / 12} \sum_{l = 0}^{n N - 1} \exp \left( \frac{\pi i}{n N} l^2 + \frac{2 \pi i}{n N} j l + \pi i l \right) Z^{k} f_{l}(z).
    \end{align}
    This formula gives that
    \begin{align}
        \chi_{\rho}(\varphi)
        &= \frac{1}{\sqrt{n N}} e^{2 \pi i (n - k) / 3 + \pi i (nN - 1) / 12} \sum_{j = 0}^{n N - 1} \exp \left( \frac{\pi i}{n N} j^2 + \frac{2 \pi i}{n N} j^2 + \pi i j \right) \\
        &= \frac{1}{\sqrt{n N}} e^{2 \pi i (n - k) / 3 + \pi i (nN - 1) / 12} \sum_{j = 0}^{n N - 1} \exp \left( \pi i \frac{3 j^{2} + n N j}{n N} \right).
    \end{align}
    A reciprocity theorem for generalized Gauss sums~\cite{Berndt-Evans-Williams1998Gauss}*{Theorem 1.2.2} implies that
    \begin{align}
        \sum_{j = 0}^{n N - 1} \exp \left( \pi i \frac{3 j^{2} + n N j}{n N} \right)
        &= \sqrt{\frac{n N}{3}} \exp \left( \frac{\pi i}{4} \left( 1 - \frac{(n N)^{2}}{3 n N} \right) \right) \sum_{l = 0}^{2} \exp \left( \pi i \frac{ -  n N l^{2} - n N l}{3} \right) \\
        &= \sqrt{\frac{n N}{3}} e^{\pi i (3 - n N) / 12} (2 + e^{- 2 \pi i n N / 3}).
    \end{align}
    Thus we have
    \begin{align}
        \chi_{\rho}(\varphi)
        &= \frac{1}{\sqrt{n N}} e^{2 \pi i (n - k) / 3 + \pi i (nN - 1) / 12} \sqrt{\frac{n N}{3}} e^{\pi i (3 - n N) / 12} (2 + e^{- 2 \pi i n N / 3}) \\
        &= \frac{1}{\sqrt{3}} e^{2 \pi i (n - k) / 3 + \pi i / 6} 
        (2 + e^{- 2 \pi i n N / 3}) \\
        &= \frac{2}{\sqrt{3}} e^{2 \pi i (n - k) / 3 + \pi i / 6} + \frac{1}{\sqrt{3}} e^{2 \pi i (n - k - n N) / 3 + \pi i / 6}.
    \end{align}
    Therefore we have 
    \begin{equation}
        \dim \mathscr{H}_{n, k}^{\varphi} = \frac{1}{3} \left( n N + \frac{4}{\sqrt{3}} \cos \left( \frac{2 \pi (n - k)}{3} + \frac{\pi}{6} \right) + \frac{2}{\sqrt{3}} \cos \left( \frac{2 \pi (n - k - n N)}{3} + \frac{\pi}{6} \right) \right).
    \end{equation}

    In the case $n < 0$, we have the orthogonal decomposition $\mathscr{H}_{n}^{\varphi} = \bigoplus_{k \geq 0} \overline{\mathscr{H}_{- n, k}^{\varphi}}$ and $\overline{\mathscr{H}_{- n, k}^{\varphi}}$ is the joint eigenspace corresponding to the eigenvalue $\sigma_{n, k}$ and $i^{- 1} T$ on $M$.
    In particular,
    \begin{align}
        \dim \overline{\mathscr{H}_{- n, k}^{\varphi}} &= \dim \mathscr{H}_{- n, k}^{\varphi} \\
        &= \frac{1}{3} \left( \abs{n} N + \frac{4}{\sqrt{3}} \cos \left( \frac{2 \pi (\abs{n} - k)}{3} + \frac{\pi}{6} \right) + \frac{2}{\sqrt{3}} \cos \left( \frac{2 \pi (\abs{n} - k - \abs{n} N)}{3} + \frac{\pi}{6} \right) \right),
    \end{align}
    which completes the proof.
\end{proof}

We next turn to the case of $\Gamma_{c}^{\prime}$.
Since $\varphi^{\prime} = (0, c / 3) \varphi$, we can use the computation in the case of $\Gamma_{c}$.

\begin{thm}
    The multiplicity $m_{n, k}^{\prime}$ of $\sigma_{n, k}$ in the joint spectrum of $\Delta_{b}$ and $i^{- 1} T$ on $M'$ is given by
    \begin{equation}
        m_{n, k}^{\prime} = \frac{1}{3} \left( \abs{n} N + \frac{4}{\sqrt{3}} \cos \left( \frac{2 \pi (\abs{n} + k)}{3} - \frac{\pi}{6} \right) + \frac{2}{\sqrt{3}} \cos \left( \frac{2 \pi (\abs{n} + k + \abs{n} N)}{3} - \frac{\pi}{6} \right) \right).
    \end{equation}
    The multiplicity of $(\kappa, 0)$ in the joint spectrum is equal to $\# E_{\kappa} / 3$ for any $\kappa > 0$.
\end{thm}

\begin{proof}
    Any function on $M^{\prime}$ can be identified with a $\varphi^{\prime}$-invariant function on $M_{\omega, c}$.
    Denote by $\mathscr{H}_{n}^{\varphi^{\prime}}$ the space of $\varphi^{\prime}$-invariant functions in $\mathscr{H}_{n}$.
    Then $L^{2}(M^{\prime}) = \bigoplus_{n \in\bbZ} \mathscr{H}_{n}^{\varphi^{\prime}}$.
    The proof for $n = 0$ is the same as that in the case of $M$.
    
    Assume that $n > 0$.
    Set $\mathscr{H}_{n, k}^{\varphi^{\prime}} = \mathscr{H}_{n, k} \cap \mathscr{H}_{n}^{\varphi^{\prime}}$, which is the space of $\varphi^{\prime}$-invariant functions in $\mathscr{H}_{n, k}$.
    Note that $\mathscr{H}_{n}^{\varphi^{\prime}} = \bigoplus_{k \in \bbZ_{\geq 0}} \mathscr{H}_{n, k}^{\varphi^{\prime}}$ and
    $\mathscr{H}_{n, k}^{\varphi^{\prime}}$ is the joint eigenspace corresponding to the eigenvalue $\sigma_{n, k}$ of $\Delta_{b}$ and $i^{- 1} T$ on $M^{\prime}$.
    To compute the dimension of $\mathscr{H}_{n, k}^{\varphi^{\prime}}$,
    we consider the representation $\rho^{\prime}$ of $\Gamma_{c}^{\prime} / \Lambda_{\omega, c} \cong \langle \varphi^{\prime} \rangle \cong \bbZ / 3 \bbZ$ on $\mathscr{H}_{n, k}$ defined by the pullback.
    The Schur orthogonality relation implies that
    \begin{equation}
        \dim \mathscr{H}_{n, k}^{\varphi^{\prime}} = \frac{1}{3} \sum_{l = 0}^{2} \chi_{\rho^{\prime}}((\varphi^{\prime})^{l}) = \frac{1}{3} (n N + 2 \Real \chi_{\rho^{\prime}}(\varphi^{\prime})).
    \end{equation}
    Since $\varphi^{\prime} = (0, c / 3) \varphi$,
    \begin{equation}
        ((\varphi^{\prime})^{\ast} f)(z, t) = f((0, c / 3) \varphi(z, t)) = e^{2 \pi i n / 3} (\varphi^{\ast} f)(z, t)
    \end{equation}
    for any $f \in \mathscr{H}_{n}$.
    Thus we have
    \begin{equation}
        \chi_{\rho^{\prime}}(\varphi^{\prime})
        = e^{2 \pi i n / 3} \chi_{\rho}(\varphi)
        = \frac{2}{\sqrt{3}} e^{- 2 \pi i (n + k) / 3 + \pi i / 6} + \frac{1}{\sqrt{3}} e^{- 2 \pi i (n + k + n N) / 3 + \pi i / 6}.
    \end{equation}
    Therefore we have 
    \begin{equation}
        \dim \mathscr{H}_{n, k}^{\varphi^{\prime}} = \frac{1}{3} \left( n N + \frac{4}{\sqrt{3}} \cos \left( \frac{2 \pi (n + k)}{3} - \frac{\pi}{6} \right) + \frac{2}{\sqrt{3}} \cos \left( \frac{2 \pi (n + k + n N)}{3} - \frac{\pi}{6} \right) \right).
    \end{equation}

    In the case $n < 0$, we have the orthogonal decomposition $\mathscr{H}_{n}^{\varphi^{\prime}} = \bigoplus_{k \geq 0} \overline{\mathscr{H}_{- n, k}^{\varphi^{\prime}}}$ and $\overline{\mathscr{H}_{- n, k}^{\varphi^{\prime}}}$ is the joint eigenspace corresponding to the eigenvalue $\sigma_{n, k}$ of $\Delta_{b}$ and $i^{- 1} T$ on $M^{\prime}$.
    In particular,
    \begin{align}
        \dim \overline{\mathscr{H}_{- n, k}^{\varphi^{\prime}}} &= \dim \mathscr{H}_{- n, k}^{\varphi^{\prime}} \\
        &= \frac{1}{3} \left( \abs{n} N + \frac{4}{\sqrt{3}} \cos \left( \frac{2 \pi (\abs{n} + k)}{3} - \frac{\pi}{6} \right) + \frac{2}{\sqrt{3}} \cos \left( \frac{2 \pi (\abs{n} + k + \abs{n} N)}{3} - \frac{\pi}{6} \right) \right),
    \end{align}
    which completes the proof.
\end{proof}

\begin{rem}
\label{rem:non conjugate of order 2}
    Consider the case $N \equiv 0$.
    Suppose to the contrary that $\Gamma_{c}$ is conjugate to $\Gamma_{c}^{\prime}$ in $\bbH \rtimes \bbC^{\ast}$.
    Then $(M, T^{1, 0} M, r \theta_{\bbH})$ is isomorphic to $(M^{\prime}, T^{1, 0} M^{\prime}, r^{\prime} \theta_{\bbH})$ for some $r^{\prime} \in \bbR_{> 0}$ as a pseudo-Hermitian manifold.
    Considering the spectrum of $i^{- 1} T$ yields that $r = r^{\prime}$,
    and so $m_{n, k}$ must be equal to $m_{n, k}^{\prime}$ for every $n$ and $k$.
    However, we can see that
    \begin{equation}
        m_{1, 0} = \frac{N}{3} - 1 \neq \frac{N}{3} = m_{1, 0}^{\prime},
    \end{equation}
    which is a contradiction.
    Therefore $\Gamma_{c}$ is not conjugate to $\Gamma_{c}^{\prime}$ in $\bbH \rtimes \bbC^{\ast}$.
\end{rem}

\section{Joint spectrum on Heisenberg Bieberbach manifolds $\Gamma \backslash \bbH$ with $\proj_{2}(\Gamma) = \langle e^{\pi i / 3} \rangle$}

To simplify notation, we write $\eta$ for $e^{\pi i / 3}$.
In this section, $\equiv$ is understood modulo $3$.
Let $\Gamma$ be a Heisenberg Bieberbach group with $\proj_{2}(\Gamma) = \langle \eta \rangle$ and consider the Heisenberg Bieberbach manifold $\Gamma \backslash \bbH$.
Recall that $N = \# (Z(\Lambda) / [\Lambda, \Lambda])$, where $\Lambda = \Gamma \cap \bbH$.
We first give a classification of such $\Gamma$ up to conjugacy in $\bbH \rtimes \bbC^{\ast}$.

\begin{prop}
\label{prop:HB group of order 6}
    Let $\Gamma$ be a Heisenberg Bieberbach group satisfying $\proj_{2}(\Gamma) = \langle \eta \rangle$ and set $c = 4 \Imaginary \eta / N$.
    Then $N$ must be even and the conjugacy class of $\Gamma$ in $\bbH \rtimes \bbC^{\ast}$ depends on $N$ modulo $3$ as follows.
    If $N \equiv 0$, then it is conjugate to either
    \begin{equation}
        \Gamma_{c} \coloneqq \langle \Lambda_{\eta, c}, \psi \coloneqq (0, c / 6) \mu(\eta) \rangle \quad \text{ or } \quad \Gamma_{c}^{\prime} \coloneqq \langle \Lambda_{\eta, c}, \psi^{\prime} \coloneqq (0, 5 c / 6) \mu(\eta) \rangle.
    \end{equation}
    If $N \equiv 1$, then it is conjugate to $\Gamma_{c}$.
    If $N \equiv 2$, then it is conjugate to $\Gamma_{c}^{\prime}$.
\end{prop}

\begin{proof}
    Without loss of generality, we may assume that $\Gamma \cap \bbH = \Lambda_{\tau, c}$.
    In addition, we can assume that $\tau = \eta$ since $L_{\tau}$ is invariant under the multiplication by $\eta$.
    Take $g = (\zeta, \sigma) \in \bbH$ such that $\gamma \coloneqq g \mu(\eta) \in \Gamma$.
    Note that $\Gamma = \langle \Lambda_{\eta, c}, \gamma \rangle$.
    We can see that
    \begin{align}
        \gamma^{2} &= (\zeta, \sigma) (\eta \zeta, \sigma) \mu(\omega) = (\zeta + \eta \zeta, 2 \sigma - 2 \abs{\zeta}^{2} \Imaginary \eta) \mu(\omega), \\
        \gamma^{3} &= (\zeta, \sigma) (\eta \zeta + \eta^{2} \zeta, 2 \sigma - 2 \abs{\zeta}^{2} \Imaginary \eta) \mu(- 1) = (2 \eta \zeta, 3 \sigma - 6 \abs{\zeta}^{2} \Imaginary \eta ) \mu(- 1),
    \end{align}
    and $\gamma^{6} = (0, 6 \sigma - 12 \abs{\zeta}^{2} \Imaginary \eta) \in \Gamma \cap \bbH = \Lambda_{\eta, c}$,
    which implies $q \coloneqq (6 \sigma - 12 \abs{\zeta}^{2} \Imaginary \eta) / c \in \bbZ$.
    Since $\Gamma \cap \proj_{2}^{-1}(\langle - 1 \rangle) = \langle \Lambda_{\eta, c}, \gamma^{3} \rangle$ is torsion-free, $N$ must be even and $q$ must be odd.
    Moreover, it follows from $\gamma \Lambda_{\eta, c} \gamma^{-1} = \Lambda_{\eta, c}$ that
    \begin{align}
        \gamma (1, 0) \gamma^{- 1} &= (\eta, - 4 \Imaginary (\eta \bar{\zeta})) = (\eta, 0) (0, - 4 \Imaginary (\eta \bar{\zeta})) \in \Lambda_{\eta, c}, \\
        \gamma (\eta, 0) \gamma^{- 1} &= (\eta^{2}, - 4 \Imaginary (\eta^{2} \bar{\zeta})) = (\eta, 0) (- 1, 0) (0, N c / 2) (0, - 4 \Imaginary (\eta^{2} \bar{\zeta})) \in \Lambda_{\eta, c},
    \end{align}
    which implies $4 \Imaginary (\eta \bar{\zeta}), 4 \Imaginary (\eta^{2} \bar{\zeta}) \in c \bbZ$.
    Hence there exist $k, l \in \bbZ$ such that
    \begin{equation}
        \zeta = \frac{c}{4 \Imaginary \eta} (k + l \eta) = \frac{k + l \eta}{N}.
    \end{equation}
    Set
    \begin{equation}
        \zeta' \coloneqq - \frac{\zeta}{1 - \eta} = \frac{l - (k + l) \eta}{N}.
    \end{equation}
    It follows from Proposition~\ref{prop:conjugate} that $h \coloneqq (\zeta', 0) \in \bbH$ satisfies $h \Lambda_{\eta, c} h^{-1} = \Lambda_{\eta, c}$.
    Moreover, we have
    \begin{equation}
        h \gamma h^{-1} = (\zeta', 0) (\zeta, \sigma) (- \eta \zeta', 0) \mu(\eta) = (0, \sigma - 2 \abs{\zeta}^{2} \Imaginary \eta) \mu(\eta) = (0, q / 6) \mu(\eta).
    \end{equation}
    Hence we may assume that $\zeta = 0$.
    It remains to determine the condition under which $\Gamma$ is torsion-free.
    If an element $\gamma'$ in $\Gamma \cap \proj_{2}^{-1}(\eta)$ or $\Gamma \cap \proj_{2}^{-1}(\eta^{5})$ is a torsion element, then so is $(\gamma')^{3} \in \Gamma \cap \proj_{2}^{-1}(- 1)$.
    Hence $\Gamma$ is torsion-free if and only if both $\Gamma \cap \proj_{2}^{-1}(\langle - 1 \rangle) = \langle \Lambda_{\eta, c}, \gamma^{3} \rangle$ and $\Gamma \cap \proj_{2}^{-1}(\langle \omega \rangle) = \langle \Lambda_{\eta, c}, \gamma^{2} \rangle$ are torsion-free.
    The former subgroup is torsion-free since $\gamma^{3} = (0, 3 \sigma) \mu(- 1)$ and $q$ is odd.
    Moreover, it follows from the proof of Proposition~\ref{prop:HB group of order 3} and $0 = (0 + (- N / 2) \omega) / N + \omega / 2$ that the latter subgroup is also torsion-free if and only if
    \begin{equation}
        q, q - \left( 0 - \frac{N}{2} + N \right), q - 2 \left( 0 - \frac{N}{2} \right) \not\equiv 0,
    \end{equation}
    or equivalently $q \not\equiv 0, 2 N$.
    Therefore $\Gamma$ is conjugate in $\bbH \rtimes \bbC^{\ast}$ to
    \begin{equation}
        \langle \Lambda_{\eta, c}, (0, \sigma) \mu(\eta) \rangle
        =
        \begin{cases}
            \Gamma_{c} & \text{if $q \equiv 1$}, \\
            \Gamma_{c}^{\prime} & \text{if $q \equiv 2$},
        \end{cases}
    \end{equation}
    which completes the proof.
\end{proof}

This proposition implies that $\Gamma \backslash \bbH$ is isomorphic to either $(M \coloneqq \Gamma_{c} \backslash \bbH, T^{1, 0} M, r \theta_{\bbH})$ or $(M^{\prime} \coloneqq \Gamma_{c}^{\prime} \backslash \bbH, T^{1, 0} M^{\prime}, r \theta_{\bbH})$ for some $r > 0$ as a pseudo-Hermitian manifold.
Hence it suffices to study the joint spectrum of $\Delta_{b}$ and $i^{- 1} T$ on these pseudo-Hermitian manifolds.
We first consider $(M, T^{1, 0} M, r \theta_{\bbH})$.

\begin{thm}
    The multiplicity $m_{n, k}$ of $\sigma_{n, k}$ in the joint spectrum of $\Delta_{b}$ and $i^{- 1} T$ on $M$ is given by
    \begin{align}
        m_{n, k}
        &= \frac{1}{6} \Bigg( \abs{n} N + 2 (- 1)^{n + k} + 2 \cos \frac{\pi (\abs{n} - k + 1)}{3}  + \frac{2}{\sqrt{3}} \cos \left( \frac{2 \pi (\abs{n} - k)}{3} + \frac{\pi}{6} \right) \\
        &\qquad \qquad+ \frac{4}{\sqrt{3}} \cos \left( \frac{2 \pi (\abs{n} - k + \abs{n} N)}{3} + \frac{\pi}{6} \right) \Bigg).
    \end{align}
    The multiplicity of $(\kappa, 0)$ in the joint spectrum is equal to $\# E_{\kappa} / 6$ for any $\kappa > 0$.
\end{thm}

\begin{proof}
    Any function on $M$ can be identified with a $\psi$-invariant function on $M_{\eta, c}$.
    Denote by $\mathscr{H}_{n}^{\psi}$ the space of $\psi$-invariant functions in $\mathscr{H}_{n}$.
    Then $L^{2}(M) = \bigoplus_{n \in\bbZ} \mathscr{H}_{n}^{\psi}$.
    
    In the case $n = 0$, one has $\mathscr{H}_{0} = \bigoplus_{\zeta \in L_{\eta}^{\prime}} \bbC \chi_{\zeta}$
    and $\Delta_{b} \chi_{\zeta} = (\pi^{2} \abs{\zeta}^{2} / r) \chi_{\zeta}$.
    Since $\psi^{\ast} \chi_{\zeta} =  \chi_{\bar{\eta} \zeta}$,
    the multiplicity of the eigenvalue $(\kappa, 0)$ in the joint spectrum of $\Delta_{b}$ and $i^{- 1} T$ on $M$ is equal to $\# E_{\kappa} / 6$ for any $\kappa > 0$.

    Assume that $n > 0$.
    Set $\mathscr{H}_{n, k}^{\psi} \coloneqq \mathscr{H}_{n, k} \cap \mathscr{H}_{n}^{\psi}$, which is the space of $\psi$-invariant functions in $\mathscr{H}_{n, k}$.
    Note that $\mathscr{H}_{n}^{\psi} = \bigoplus_{k \in \bbZ_{\geq 0}} \mathscr{H}_{n, k}^{\psi}$ and
    $\mathscr{H}_{n, k}^{\psi}$ is the joint eigenspace corresponding to the eigenvalue $\sigma_{n, k}$ of $\Delta_{b}$ and $i^{- 1} T$ on $M$.
    To compute the dimension of $\mathscr{H}_{n, k}^{\psi}$,
    we consider the representation $\rho$ of $\Gamma_{c} / \Lambda_{\eta, c} \cong \langle \psi \rangle \cong \bbZ / 6 \bbZ$ on $\mathscr{H}_{n, k}$ defined by the pullback.
    The Schur orthogonality relation implies that
    \begin{equation}
        \dim \mathscr{H}_{n, k}^{\psi} = \frac{1}{6} \sum_{l = 0}^{5} \chi_{\rho}(\psi^{l}) = \frac{1}{6} (n N + 2 \Real \chi_{\rho}(\psi) + 2 \Real \chi_{\rho}(\psi^{2}) + \chi_{\rho}(\psi^{3})).
    \end{equation}
    As we showed in the proof of Theorem~\ref{thm:joint spectrum of order 2}, $\chi_{\rho}(\psi^{3}) = \chi_{\rho}(\varphi) = 2 (- 1)^{n + k}$, where $\varphi$ is as introduced in Proposition~\ref{prop:HB group of order 2}.
    Hence it suffices to compute the values $\chi_{\rho}(\psi)$ and $\chi_{\rho}(\psi^{2})$.
    
    We first consider $\chi_{\rho}(\psi)$.
    For each $0 \leq j \leq n N - 1$, it follows from $N = 4 \Imaginary \eta / c$ that
    \begin{align}
        (\psi^{\ast} f_{j})(z, t)
        &= f_{j}(\eta z, t + c / 6) \\
        & = \exp \left( \frac{2 \pi n}{c} \left(i t + \frac{i c}{6} - \abs{\eta z}^{2} + (\eta z)^{2} \right) \right) h_{j}(\eta z) \\
        &= e^{\pi i n / 3} \exp \left( \frac{2 \pi n}{c} \left(i t - \abs{z}^{2} + \eta^{2} z^{2} \right) \right) h_{j}(\eta z) \\
        &= e^{\pi i n / 3} \exp \left( \frac{2 \pi n}{c} \left(i t - \abs{z}^{2} + z^2 + \eta (\eta - \overline{\eta}) z^{2} \right) \right) h_{j}(\eta z) \\
        &= e^{\pi i n / 3} \exp \left( \frac{2 \pi n}{c} (i t - \abs{z}^{2} + z^{2}) \right) \exp \left( \pi i n N \eta z^{2} \right) h_{j}(\eta z).
    \end{align}
    Now we have
        \begin{equation}\label{h_j-eta}
            \exp \left( \pi i n N \eta z^{2} \right) h_{j}(\eta z)
            = \sum_{m \in \bbZ} \exp \left(
                \pi i n N \eta \left( m + \frac{j}{nN} + z \right)^2 \right).
        \end{equation}
    Consider $g(x) \coloneqq \exp(\pi i n N \eta (x + j / n N + z)^{2}) \in \mathcal{S}(\bbR)$.
    Then the equation \eqref{h_j-eta} and the Poisson summation formula give
    \begin{equation}
        \exp \left( \pi i n N \eta z^{2} \right) h_{j}(\eta z) = \sum_{m \in \bbZ} g(m) = \sum_{l = 0}^{n N - 1} \sum_{m^{\prime} \in \bbZ} \hat{g}(n N m^{\prime} + l),
    \end{equation}
    where $\hat{g}$ is the Fourier transform of $g$.
    A computation yields
    \begin{align}
        \hat{g}(\xi)
        &= \int_{\bbR} \exp \left( \pi i n N \eta \left( x + \frac{j}{n N} + z \right)^{2} - 2 \pi i x \xi \right) \, d x \\
        &= \frac{1}{\sqrt{n N}} e^{\pi i / 12} \exp \left( - \frac{\pi i}{n N \eta} \xi^{2} + 2 \pi i \left( \frac{j}{n N} + z \right) \xi \right) \\
        &= \frac{1}{\sqrt{n N}} e^{\pi i / 12} \exp \left( \frac{\pi i (\eta - 1)}{n N} \xi^{2} + 2 \pi i \left( \frac{j}{n N} + z \right) \xi \right),
    \end{align}
    which implies
    \begin{align}
        &\exp \left( \pi i n N \eta z^{2} \right) h_{j}(\eta z) \\
        &= \frac{1}{\sqrt{n N}} e^{\pi i / 12} \sum_{l = 0}^{n N - 1} \sum_{m^{\prime}} \exp \left( \frac{\pi i (\eta - 1)}{n N} (n N m^{\prime} + l)^{2} + 2 \pi i \left( \frac{j}{n N} + z \right) (n N m^{\prime} + l) \right) \\
        &= \frac{1}{\sqrt{n N}} e^{\pi i / 12} \sum_{l = 0}^{n N - 1} \exp \left( - \frac{\pi i}{n N} l^2 + \frac{2 \pi i}{n N} j l \right) h_{l}(z).
    \end{align}
    Hence we have
    \begin{align}
        (\psi^{\ast} f_{j})(z, t)
        &= e^{\pi i n / 3} \exp \left( \frac{2 \pi n}{c} (i t - \abs{z}^{2} + z^{2}) \right) \\
        &\quad \times \frac{1}{\sqrt{n N}} e^{\pi i / 12} \sum_{l = 0}^{n N - 1} \exp \left( - \frac{\pi i}{n N} l^2 + \frac{2 \pi i}{n N} j l \right) h_{l}(z) \\
        &= \frac{1}{\sqrt{n N}} e^{\pi i n / 3 + \pi i / 12} \sum_{l = 0}^{n N - 1} \exp \left( - \frac{\pi i}{n N} l^2 + \frac{2 \pi i}{n N} j l\right) f_{l}(z).
    \end{align}
    Moreover, one can see that
    \begin{equation}
        (\psi^{\ast} (Z f))(z, t) = (Z f)(\eta z, t + c / 6) = e^{- \pi i / 3} (Z (\psi^{\ast} f))(z, t)
    \end{equation}
    for any $f \in C^{\infty}(\bbH)$.
    Thus we have
    \begin{align}
        \psi^{\ast} (Z^{k} f_{j})
        &= e^{- \pi i k / 3} Z^{k} (\psi^{\ast} f_j) \\
        &= \frac{1}{\sqrt{n N}} e^{\pi i (n - k) / 3 + \pi i / 12} \sum_{l = 0}^{n N - 1} \exp \left( -\frac{\pi i}{n N} l^2 + \frac{2 \pi i}{n N} j l \right) Z^{k} f_{l}(z).
    \end{align}
    This formula gives that
    \begin{align}
        \chi_{\rho}(\psi)
        &= \frac{1}{\sqrt{n N}} e^{\pi i (n - k) / 3 + \pi i / 12} \sum_{j = 0}^{n N - 1} \exp \left(  \frac{\pi i}{n N} j^2 \right) \\
        &= \frac{1}{\sqrt{n N}} e^{\pi i (n - k) / 3 + \pi i / 12} \sqrt{n N} e^{\pi i / 4} \\
        &= e^{\pi i (n - k + 1) / 3}.
    \end{align}
    Here, we use a result on the quadratic Gauss sum~\cite{Lang1994ANT}*{Chapter IV.3}.
    
    We next consider $\chi_{\rho}(\psi^{2})$.
    It follows from the formula of $\psi^{\ast} (Z^{k} f_{j})$ that
    \begin{align}
        (\psi^{2})^{\ast} (Z^{k} f_{j})
        &= \frac{1}{\sqrt{n N}} e^{\pi i (n - k) / 3 + \pi i / 12} \sum_{l = 0}^{n N - 1} \exp \left( -\frac{\pi i}{n N} l^2 + \frac{2 \pi i}{n N} j l \right) \psi^{\ast} (Z^{k} f_{l}) \\
        &= \frac{1}{n N} e^{2 \pi i (n - k) / 3 + \pi i / 6} \sum_{l = 0}^{n N - 1} \sum_{m = 0}^{n N - 1} \exp \left( -\frac{\pi i}{n N} l^2 + \frac{2 \pi i}{n N} j l  -\frac{\pi i}{n N} m^2 + \frac{2 \pi i}{n N} l m  \right) Z^{k} f_{m}.
    \end{align}
    Hence we have
    \begin{align}
        \chi_{\rho}(\psi^{2})
        &= \frac{1}{n N} e^{2 \pi i (n - k) / 3 + \pi i / 6}  \sum_{j = 0}^{n N - 1} \sum_{l = 0}^{n N - 1} \exp \left( -\frac{\pi i}{n N} l^2 + \frac{2 \pi i}{n N} j l -\frac{\pi i}{n N} j^2 + \frac{2 \pi i}{n N} j l \right) \\
        &= \frac{1}{n N} e^{2 \pi i (n - k) / 3 + \pi i / 6}  \sum_{j = 0}^{n N - 1} \exp \left( -\frac{\pi i}{n N} j^2 \right) \sum_{l = 0}^{n N - 1} \exp \left( \pi i \frac{- l^{2} + 4 j l}{n N} \right).
    \end{align}
    A reciprocity theorem for generalized Gauss sums~\cite{Berndt-Evans-Williams1998Gauss}*{Theorem 1.2.2} implies that
    \begin{equation}
        \sum_{l = 0}^{n N - 1} \exp \left( \pi i \frac{- l^{2} + 4 j l}{n N} \right) = \sqrt{n N} \exp \left( \frac{\pi i}{4} \left( - 1 + \frac{(4 j)^{2}}{n N} \right) \right) = \sqrt{n N} e^{- \pi i / 4} \exp \left( \frac{4 \pi i}{n N} j^2\right),
    \end{equation}
    and hence
    \begin{align}
        \chi_{\rho}(\psi^{2})
        &= \frac{1}{n N} e^{2 \pi i (n - k) / 3 + \pi i / 6}  \sum_{j = 0}^{n N - 1} \exp \left( -\frac{\pi i}{n N} j^2 \right) \sqrt{n N} e^{- \pi i / 4} \exp \left( \frac{4 \pi i}{n N} j^2 \right) \\
        &= \frac{1}{\sqrt{n N}} e^{2 \pi i (n - k) / 3 - \pi i / 12} \sum_{j = 0}^{n N - 1} \exp \left( \frac{3 \pi i}{n N} j^2 \right). \\
    \end{align}
    By applying the same theorem again, we have
    \begin{align}
        \sum_{j = 0}^{n N - 1} \exp \left( \frac{3 \pi i}{n N} j^2 \right) &= \sqrt{\frac{n N}{3}} e^{\pi i / 4} \sum_{l = 0}^{2} \exp 
        \left( - \frac{ \pi i n N}{3} l^2 \right) \\
        &= \sqrt{\frac{n N}{3}} e^{\pi i / 4} (1 + e^{- \pi i n N/3} + e^{-4 \pi i n N/3}) \\
        &= \sqrt{\frac{n N}{3}} e^{\pi i / 4} (1 + 2 e^{2 \pi i n N / 3});
    \end{align}
    here we use the fact that $N$ is even.
    It follows that 
    \begin{align}
        \chi_{\rho}(\psi^{2})
        &= \frac{1}{\sqrt{n N}} e^{2 \pi i (n - k) / 3 - \pi i / 12} \sqrt{\frac{n N}{3}} e^{\pi i / 4} (1 + 2 e^{2 \pi i n N / 3}) \\
        &= \frac{1}{\sqrt{3}} e^{2 \pi i (n - k) / 3 + \pi i / 6} + \frac{2}{\sqrt{3}} e^{2 \pi i (n - k + n N) / 3 + \pi i / 6}.
    \end{align}
    Therefore we have 
    \begin{align}
        \dim \mathscr{H}_{n, k}^{\psi}
        &= \frac{1}{6} \Bigg( n N + 2 (- 1)^{n + k} + 2 \cos \frac{\pi (n - k + 1)}{3}  + \frac{2}{\sqrt{3}} \cos \left( \frac{2 \pi (n - k)}{3} + \frac{\pi}{6} \right) \\
        &\qquad \qquad+ \frac{4}{\sqrt{3}} \cos \left( \frac{2 \pi (n - k + n N)}{3} + \frac{\pi}{6} \right) \Bigg).
    \end{align}

    In the case $n < 0$, we have the orthogonal decomposition $\mathscr{H}_{n}^{\psi} = \bigoplus_{k \geq 0} \overline{\mathscr{H}_{- n, k}^{\psi}}$ and $\overline{\mathscr{H}_{- n, k}^{\psi}}$ is the joint eigenspace corresponding to the eigenvalue $\sigma_{n, k}$ of $\Delta_{b}$ and $i^{- 1} T$ on $M$.
    In particular,
    \begin{align}
        \dim \overline{\mathscr{H}_{- n, k}^{\psi}}
        &= \dim \mathscr{H}_{- n, k}^{\psi} \\
        &= \frac{1}{6} \Bigg( \abs{n} N + 2 (- 1)^{n + k} + 2 \cos \frac{\pi (\abs{n} - k + 1)}{3}  + \frac{2}{\sqrt{3}} \cos \left( \frac{2 \pi (\abs{n} - k)}{3} + \frac{\pi}{6} \right) \\
        &\qquad \qquad+ \frac{4}{\sqrt{3}} \cos \left( \frac{2 \pi (\abs{n} - k + \abs{n} N)}{3} + \frac{\pi}{6} \right) \Bigg),
    \end{align}
    which completes the proof.
\end{proof}

We next turn to the case of $\Gamma_{c}^{\prime}$.
Since $\psi^{\prime} = (0, 2 c / 3) \psi$, we can use the computation in the case of $\Gamma_{c}$.

\begin{thm}
    The multiplicity $m_{n, k}^{\prime}$ of $\sigma_{n, k}$ in the joint spectrum of $\Delta_{b}$ and $i^{- 1} T$ on $M'$ is given by
    \begin{align}
        m_{n, k}^{\prime}
        &= \frac{1}{6} \Bigg( \abs{n} N + 2 (- 1)^{n + k} + 2 \cos \frac{\pi (\abs{n} + k - 1)}{3} + \frac{2}{\sqrt{3}} \cos \left( \frac{2 \pi (\abs{n} + k)}{3} - \frac{\pi}{6} \right) \\
        &\qquad \qquad+ \frac{4}{\sqrt{3}} \cos \left( \frac{2 \pi (\abs{n} + k - \abs{n} N)}{3} - \frac{\pi}{6} \right) \Bigg).
    \end{align}
    The multiplicity of $(\kappa, 0)$ in the joint spectrum is equal to $\# E_{\kappa} / 6$ for any $\kappa > 0$.
\end{thm}

\begin{proof}
    Any function on $M^{\prime}$ can be identified with a $\psi^{\prime}$-invariant function on $M_{\omega, c}$.
    Denote by $\mathscr{H}_{n}^{\psi^{\prime}}$ the space of $\psi^{\prime}$-invariant functions in $\mathscr{H}_{n}$.
    Then $L^{2}(M^{\prime}) = \bigoplus_{n \in\bbZ} \mathscr{H}_{n}^{\psi^{\prime}}$.
    The proof for $n = 0$ is the same as that in the case of $M$.
    
    Assume that $n > 0$.
    Set $\mathscr{H}_{n, k}^{\psi^{\prime}} \coloneqq \mathscr{H}_{n, k} \cap \mathscr{H}_{n}^{\psi^{\prime}}$, which is the space of $\psi^{\prime}$-invariant functions in $\mathscr{H}_{n, k}$.
    Note that $\mathscr{H}_{n}^{\psi^{\prime}} = \bigoplus_{k \in \bbZ_{\geq 0}} \mathscr{H}_{n, k}^{\psi^{\prime}}$ and
    $\mathscr{H}_{n, k}^{\psi^{\prime}}$ is the joint eigenspace corresponding to the eigenvalue $\sigma_{n, k}$ of $\Delta_{b}$ and $i^{- 1} T$ on $M^{\prime}$.
    To compute the dimension of $\mathscr{H}_{n, k}^{\psi^{\prime}}$,
    we consider the representation $\rho^{\prime}$ of $\Gamma_{c}^{\prime} / \Lambda_{\omega, c} \cong \langle \psi^{\prime} \rangle \cong \bbZ / 6 \bbZ$ on $\mathscr{H}_{n, k}$ defined by the pullback.
    The Schur orthogonality relation implies that
    \begin{equation}
        \dim \mathscr{H}_{n, k}^{\psi^{\prime}} = \frac{1}{6} \sum_{l = 0}^{5} \chi_{\rho^{\prime}}((\psi^{\prime})^{l}) = \frac{1}{6} (n N + 2 \Real \chi_{\rho^{\prime}}(\psi^{\prime}) + 2 \Real \chi_{\rho^{\prime}}((\psi^{\prime})^{2}) + \chi_{\rho^{\prime}}((\psi^{\prime})^{3})).
    \end{equation}
    Since $\psi^{\prime} = (0, 2 c / 3) \psi$,
    \begin{equation}
        ((\psi^{\prime})^{\ast} f)(z, t) = f((0, 2 c / 3) \psi(z, t)) = e^{4 \pi i n / 3} (\psi^{\ast} f)(z, t)
    \end{equation}
    for any $f \in \mathscr{H}_{n}$.
    Thus we have
    \begin{align}
        \chi_{\rho^{\prime}}(\psi^{\prime}) &= e^{4 \pi i n / 3} \chi_{\rho}(\psi) = e^{- \pi i (n + k - 1) / 3}, \\
        \chi_{\rho^{\prime}}((\psi^{\prime})^{2}) &= e^{8 \pi i n / 3} \chi_{\rho}(\psi^{2}) = \frac{1}{\sqrt{3}} e^{- 2 \pi i (n + k) / 3 + \pi i / 6} + \frac{2}{\sqrt{3}} e^{- 2 \pi i (n + k - n N) / 3 + \pi i / 6},
    \end{align}
    and $\chi_{\rho^{\prime}}((\psi^{\prime})^{3}) = \chi_{\rho}(\psi^{3}) = 2 (- 1)^{n + k}$.
    These formulas give that 
    \begin{align}
        \dim \mathscr{H}_{n, k}^{\psi^{\prime}}
        &= \frac{1}{6} \Bigg( n N + 2 (- 1)^{n + k} + 2 \cos \frac{\pi (n + k - 1)}{3} + \frac{2}{\sqrt{3}} \cos \left( \frac{2 \pi (n + k)}{3} - \frac{\pi}{6} \right) \\
        &\qquad \qquad+ \frac{4}{\sqrt{3}} \cos \left( \frac{2 \pi (n + k - n N)}{3} - \frac{\pi}{6} \right) \Bigg).
    \end{align}

    In the case $n < 0$, we have the orthogonal decomposition $\mathscr{H}_{n}^{\psi^{\prime}} = \bigoplus_{k \geq 0} \overline{\mathscr{H}_{- n, k}^{\psi^{\prime}}}$ and $\overline{\mathscr{H}_{- n, k}^{\psi^{\prime}}}$ is the joint eigenspace corresponding to the eigenvalue $\sigma_{n, k}$ of $\Delta_{b}$ and $i^{- 1} T$ on $M^{\prime}$.
    In particular,
    \begin{align}
        \dim \overline{\mathscr{H}_{- n, k}^{\psi^{\prime}}}
        &= \dim \mathscr{H}_{- n, k}^{\psi^{\prime}} \\
        &= \frac{1}{6} \Bigg( \abs{n} N + 2 (- 1)^{n + k} + 2 \cos \frac{\pi (\abs{n} + k - 1)}{3} + \frac{2}{\sqrt{3}} \cos \left( \frac{2 \pi (\abs{n} + k)}{3} - \frac{\pi}{6} \right) \\
        &\qquad \qquad+ \frac{4}{\sqrt{3}} \cos \left( \frac{2 \pi (\abs{n} + k - \abs{n} N)}{3} - \frac{\pi}{6} \right) \Bigg),
    \end{align}
    which completes the proof.
\end{proof}

\begin{rem}
\label{rem:non conjugate of order 6}
    Consider the case $N \equiv 0$.
    Suppose to the contrary that $\Gamma_{c}$ is conjugate to $\Gamma_{c}^{\prime}$ in $\bbH \rtimes \bbC^{\ast}$.
    Then $(M, T^{1, 0} M, r \theta_{\bbH})$ is isomorphic to $(M^{\prime}, T^{1, 0} M^{\prime}, r^{\prime} \theta_{\bbH})$ for some $r^{\prime} \in \bbR_{> 0}$ as a pseudo-Hermitian manifold.
    Considering the spectrum of $i^{- 1} T$ yields that $r = r^{\prime}$,
    and so $m_{n, k}$ must be equal to $m_{n, k}^{\prime}$ for every $n$ and $k$.
    However, we can see that
    \begin{equation}
        m_{1, 0} = \frac{N}{6} - 1 \neq \frac{N}{6} = m_{1, 0}^{\prime},
    \end{equation}
    which is a contradiction.
    Therefore $\Gamma_{c}$ is not conjugate to $\Gamma_{c}^{\prime}$ in $\bbH \rtimes \bbC^{\ast}$.
\end{rem}

\section*{Acknowledgments}

The authors are grateful to Professor Yoshinobu Kamishima for helpful suggestions regarding the proof of Lemma~\ref{lem:rotation}.
Part of this work was carried out while the first author was visiting University of Tsukuba. The first author would like to thank University of Tsukuba for its hospitality. The first author is also grateful to Professor Yoshihiko~Matsumoto for valuable discussions.

\bibliographystyle{plain}
\bibliography{myrefs}

@article {Ponge2008HC,
    AUTHOR = {Ponge, Rapha\"el S.},
     TITLE = {Heisenberg calculus and spectral theory of hypoelliptic
              operators on {H}eisenberg manifolds},
   JOURNAL = {Mem. Amer. Math. Soc.},
  FJOURNAL = {Memoirs of the American Mathematical Society},
    VOLUME = {194},
      YEAR = {2008},
    NUMBER = {906},
     PAGES = {viii+ 134},
      ISSN = {0065-9266,1947-6221},
   MRCLASS = {58J40 (35H10)},
  MRNUMBER = {2417549},
MRREVIEWER = {Fabio\ Nicola},
       DOI = {10.1090/memo/0906},
       URL = {https://doi-org.tsukuba.idm.oclc.org/10.1090/memo/0906},
}

@misc{Suzuki2024preprint,
Author = {Yoshiaki Suzuki},
Title = {The {F}olland-{S}tein spectrum of some {H}eisenberg {B}ieberbach manifolds},
Year = {2024},
Eprint = {arXiv:2402.15093},
Note = {\texttt{arXiv:2402.15093}, to appear in Kyushu J. Math.},
}

@article {Jerison-Lee1987Yamabe,
    AUTHOR = {Jerison, David and Lee, John M.},
     TITLE = {The {Y}amabe problem on {CR} manifolds},
   JOURNAL = {J. Differential Geom.},
  FJOURNAL = {Journal of Differential Geometry},
    VOLUME = {25},
      YEAR = {1987},
    NUMBER = {2},
     PAGES = {167--197},
      ISSN = {0022-040X,1945-743X},
   MRCLASS = {58G30 (53C15)},
  MRNUMBER = {880182},
MRREVIEWER = {Dennis\ M.\ DeTurck},
       URL = {http://projecteuclid.org.tsukuba.idm.oclc.org/euclid.jdg/1214440849},
}

@book {Folland1989Harmonic,
    AUTHOR = {Folland, Gerald B.},
     TITLE = {Harmonic analysis in phase space},
    SERIES = {Annals of Mathematics Studies},
    VOLUME = {122},
 PUBLISHER = {Princeton University Press, Princeton, NJ},
      YEAR = {1989},
     PAGES = {x+277},
      ISBN = {0-691-08527-7; 0-691-08528-5},
   MRCLASS = {22E30 (43A80 58G15 81S30)},
  MRNUMBER = {983366},
       DOI = {10.1515/9781400882427},
       URL = {https://doi-org.tsukuba.idm.oclc.org/10.1515/9781400882427},
}

@book {Berndt-Evans-Williams1998Gauss,
    AUTHOR = {Berndt, Bruce C. and Evans, Ronald J. and Williams, Kenneth
              S.},
     TITLE = {Gauss and {J}acobi sums},
    SERIES = {Canadian Mathematical Society Series of Monographs and
              Advanced Texts},
      NOTE = {A Wiley-Interscience Publication},
 PUBLISHER = {John Wiley \& Sons, Inc., New York},
      YEAR = {1998},
     PAGES = {xii+583},
      ISBN = {0-471-12807-4},
   MRCLASS = {11L05 (11A15 11L10 11T22 11T24)},
  MRNUMBER = {1625181},
MRREVIEWER = {Charles\ Helou},
}

@book {Lang1994ANT,
    AUTHOR = {Lang, Serge},
     TITLE = {Algebraic number theory},
    SERIES = {Graduate Texts in Mathematics},
    VOLUME = {110},
   EDITION = {Second},
 PUBLISHER = {Springer-Verlag, New York},
      YEAR = {1994},
     PAGES = {xiv+357},
      ISBN = {0-387-94225-4},
   MRCLASS = {11Rxx (11-01 11-02)},
  MRNUMBER = {1282723},
MRREVIEWER = {M.\ Ram\ Murty},
       DOI = {10.1007/978-1-4612-0853-2},
       URL = {https://doi-org.tsukuba.idm.oclc.org/10.1007/978-1-4612-0853-2},
}

@article {burns-shnider1976spherical,
    AUTHOR = {Burns, Jr., D. and Shnider, S.},
     TITLE = {Spherical hypersurfaces in complex manifolds},
   JOURNAL = {Invent. Math.},
  FJOURNAL = {Inventiones Mathematicae},
    VOLUME = {33},
      YEAR = {1976},
    NUMBER = {3},
     PAGES = {223--246},
      ISSN = {0020-9910,1432-1297},
   MRCLASS = {32M10 (32F99 32K99)},
  MRNUMBER = {419857},
MRREVIEWER = {Shoshichi\ Kobayashi},
       DOI = {10.1007/BF01404204},
       URL = {https://doi-org.tsukuba.idm.oclc.org/10.1007/BF01404204},
}

@article {folland2004compact,
    AUTHOR = {Folland, G. B.},
     TITLE = {Compact {H}eisenberg manifolds as {CR} manifolds},
   JOURNAL = {J. Geom. Anal.},
  FJOURNAL = {The Journal of Geometric Analysis},
    VOLUME = {14},
      YEAR = {2004},
    NUMBER = {3},
     PAGES = {521--532},
      ISSN = {1050-6926,1559-002X},
   MRCLASS = {32V15 (14K99 32V20)},
  MRNUMBER = {2077163},
MRREVIEWER = {Thomas\ Garrity},
       DOI = {10.1007/BF02922102},
       URL = {https://doi.org/10.1007/BF02922102},
}

@book {charlap1986bieberbach,
    AUTHOR = {Charlap, Leonard S.},
     TITLE = {Bieberbach groups and flat manifolds},
    SERIES = {Universitext},
 PUBLISHER = {Springer-Verlag, New York},
      YEAR = {1986},
     PAGES = {xiv+242},
      ISBN = {0-387-96395-2},
   MRCLASS = {57S30 (22E40 53C30)},
  MRNUMBER = {862114},
MRREVIEWER = {Kyung\ Bai\ Lee},
       DOI = {10.1007/978-1-4613-8687-2},
       URL = {https://doi.org/10.1007/978-1-4613-8687-2},
}

\end{document}